\documentclass{amsproc}
\usepackage{amssymb}
\usepackage{amsthm}
\usepackage{amsfonts}
\usepackage{amsmath}
\usepackage{multirow}
\usepackage{subfigure}
\usepackage{epsfig}
\usepackage{epstopdf}
\usepackage{array}
\usepackage{hyperref}
\usepackage{verbatim}
\usepackage{float}
\usepackage{color}

\restylefloat{figure}
\theoremstyle{plain}
\newtheorem{lemma}{Lemma}[section]
\newtheorem{corollary}[lemma]{Corollary}
\newtheorem{definition}[lemma]{Definition}
\newtheorem{example}[lemma]{Example}

\newtheorem{proposition}[lemma]{Proposition}
\newtheorem{remark}[lemma]{Remark}
\newtheorem{theorem}[lemma]{Theorem}
\numberwithin{equation}{section}
\numberwithin{figure}{section}
\DeclareMathOperator{\fS}{\mathfrak{S}}
\DeclareMathOperator{\card}{card}
\DeclareMathOperator{\I}{\mathcal{I}}
\DeclareMathOperator{\J}{\mathcal{J}}
\DeclareMathOperator{\Prob}{\mathbb{P}}
\DeclareMathOperator{\inte}{int}
\DeclareMathOperator{\diam}{diam}
\DeclareMathOperator{\F}{\mathcal{F}}
\DeclareMathOperator{\C}{\mathcal{C}}
\DeclareMathOperator{\BbE}{\mathbb{E}}
\DeclareMathOperator{\BbN}{\mathbb{N}}
\DeclareMathOperator{\cN}{\mathcal{N}}
\DeclareMathOperator{\BbR}{\mathbb{R}}
\DeclareMathOperator{\dimlc}{\dim_{loc}}
\DeclareMathOperator{\dimulc}{\overline{\dim}_{loc}}
\DeclareMathOperator{\dimllc}{\underline{\dim}_{loc}}
\DeclareMathOperator{\supp}{supp}

\DeclareMathOperator{\E}{\mathbb{E}}

\begin{document}
\title[Dimensions of Random measures]{Local dimensions of random homogeneous
self-similar measures: strong separation and finite type}
\author[K.~E.~Hare, K.~G.~Hare, S.~Troscheit]{Kathryn E.~Hare, Kevin G.
Hare, Sascha Troscheit}
\thanks{Research of K.~E.~Hare was supported by NSERC Grant 2016-03719}
\thanks{Research of K.~G.~Hare was supported by NSERC Grant 2014-03154}
\thanks{Research of S.~Troscheit was supported by NSERC Grants 2016-03719,
2014-03154, and the University of Waterloo}
\address{Dept.\ of Pure Mathematics, University of Waterloo, Waterloo, Ont.,
N2L 3G1, Canada}
\email{kehare@uwaterloo.ca}
\address{Dept.\ of Pure Mathematics, University of Waterloo, Waterloo, Ont.,
N2L 3G1, Canada}
\email{kghare@uwaterloo.ca}
\address{Dept.\ of Pure Mathematics, University of Waterloo, Waterloo, Ont.,
N2L 3G1, Canada}
\email{stroscheit@uwaterloo.ca}

\begin{abstract}
We study the multifractal analysis of self-similar measures arising from
random homogeneous iterated function systems. Under the assumption of the
uniform strong separation condition, we see that this analysis parallels
that of the deterministic case. The overlapping case is more complicated; we
introduce the notion of finite type for random homogeneous iterated function
systems and give a formula for the local dimensions of finite type, regular,
random homogeneous self-similar measures in terms of Lyapunov exponents of
certain transition matrices. We show that almost all points with respect to
this measure are described by a distinguished subset called the essential
class, and that the dimension of the support can be computed almost surely
from knowledge of this essential class. For a special subcase, that we call
commuting, we prove that the set of attainable local dimensions is almost
surely a closed interval. Particular examples of such random measures are
analyzed in more detail.
\end{abstract}

\maketitle

\section{Introduction}

In this paper we study the multifractal analysis of random homogeneous
self-similar measures arising from random homogeneous iterated function
systems (RIFS) under various separation conditions.

For the case of a self-similar measure arising from a single iterated
function system (IFS), satisfying a suitable separation condition, the
multifractal analysis is well understood. The set of local dimensions of the
measure is a closed interval whose endpoints are simple to compute and there
is a formula for the Hausdorff dimensions of the set of points whose local
dimension is a given value. We refer the reader to \cite%
{FalcTech,FractalGeo3} for more details.

If, instead, the IFS has `overlaps', the multifractal analysis is more
poorly understood and can be quite different. For instance, there can be an
isolated point in the set of local dimensions of the associated self-similar
measures; see \cite{HHM,HHN,HL,Sh} for examples. A weaker notion than the
open set condition is the finite type property. This is stronger than the
weak separation condition, but satisfied by many self-similar measures which
fail to possess the open set condition, including Bernoulli convolutions
with contraction factor the reciprocal of a Pisot number and $m$-fold
convolutions of Cantor measures on Cantor sets with ratio of dissection
equal to $1/d$ for an integer $d$ and $m\geq d$; see \cite{HHM,NW}. Building
upon the foundational work of Feng in \cite{F3,F1,F2}, two of the authors,
with various coauthors, in \cite{HHM,HHN,HHS} developed a general machinery
for studying the local dimension theory of self-similar measures of finite
type.

A variant of the notion of a deterministically self-similar set is a random
self-similar set. Such sets arise as the attractor of a random iterated
function system where one begins with a finite collection of iterated
function systems, each consisting of finitely many contractions, and then
proceeds with an iterative construction where the choice of contraction to
use at each step is determined by some random process. When the contractions are
chosen independently for each cylinder in the construction, the process is
called random recursive or $\infty $-variable. If one, instead, chooses the
IFS independently for each level of the construction, but applies it to all
cylinders homogeneously, the process is called random homogeneous or $1$%
-variable for short. The alternative names arise from a construction called
the $V$-variable process that attempts to interpolate between the random
homogeneous and random recursive processes. An appropriate application of
the contraction mapping principle can be used to show there is an invariant
measure whose support is the random attractor, known as a self-similar
measure.

Random recursive sets were first studied independently by Falconer~\cite%
{Falconer86} and Graf~\cite{Graf87}, who determined their almost sure
Hausdorff dimension and measure properties. It is given by the random
analogue of the similarity dimension, the unique exponent such that
expectation of the Hutchinson--Moran sums is one. Olsen~\cite{Olsen94} and
Arbeiter and Patzschke~\cite{Arbeiter96} independently considered the
question of the almost sure multifractal spectrum under a random analogue of
the open set condition. The former considered the general case of
graph-directed random recursive constructions, whereas the latter considered
the standard random recursive model, but under weaker assumptions than
Olsen. Both found that the multifractal spectrum is almost surely the
natural variant of the deterministic case; it is given by the Legendre
transform of an implicitly defined function that satisfies the deterministic
condition on average.

Here we focus on $1$-variable, random iterated function systems and their
associated random homogeneous self-similar measures. These fractals arise by
starting with a finite collection of iterated function systems $\fS_{j}$, $%
j=1,\dots,m$, on $\mathbb{R}^{d}$ and a probability vector which specifies
the likelihood of choosing $\fS_{j}$ at a given level. Consider $\fS%
_{i}(K)=\bigcup_{j}S_{i,j}(K)$ as an operator on the space of compact
subsets of $\BbR^{d}$ with respect to the Hausdorff metric. The associated
random self-similar set or attractor is the set 
\begin{equation*}
K_{\omega }=\bigcap_{k=1}^{\infty }\fS_{\omega _{1}}\circ \fS_{\omega
_{2}}\circ \dots \circ \fS_{\omega _{k}}(K),
\end{equation*}
where $\omega =(\omega _{i})$, $\omega _{i}\in \{1,\dots,m\}$ and $K$ is a
sufficiently large non-empty compact subset of $\mathbb{R}^{d}$. By `large
enough' we mean $\fS_{j}(K)\subseteq K$ for all $j$. Each random set $%
K_{\omega }$ supports a family of random self-similar measures $\mu _{\omega
}$ that are invariant in a suitable sense; see Section \ref{sec:2} for more
details.

Most of the research on these random fractals has been on the dimensional
properties of the random sets such as in \cite{Ha,Roy11,Troscheit17}, or on
the multifractal analysis of special examples, such as the Sierpi\'{n}ski
carpets and sponges in \cite{Fraser11, Olsen11}. As with deterministic
self-similar sets, the box-counting and Hausdorff dimensions coincide (at
least, almost surely) and under certain separation conditions there is a
formula for the dimension in terms of the contraction factors of the
underlying IFS and the probability vector; see (\ref{Hdim}). But unlike the
deterministic case, the Hausdorff measure of $K_{\omega }$ is typically zero.

The goal of this paper is to investigate the local dimensional properties of
the random self-similar measures arising in the random homogeneous case.
First, we study the multifractal analysis for these measures under the
assumption of the uniform strong separation condition. We note that 
\cite{Li12,Wu11} gave
partial results in that direction and our method is similar
to that given in \cite{FalcTech, FractalGeo3} for the deterministic case
under the corresponding assumption. However, we are dealing with a slightly different framework and 
technical complications arise.
Analogously to the deterministic case, the set of attainable local
dimensions is almost surely a closed interval, formulas are given for the
endpoints and the Hausdorff dimensions of the sets where a given local
dimension is attained can be computed from the Legendre transform of a
suitable concave function. This is done in Section \ref{sec:3} of the paper.

In Section \ref{sec:equicontractivefinitetype} we introduce the finite type
condition for this random model under the assumption that all the
similarities in the RIFS have the same contraction factor. We study the
geometric structure of such random attractors, including the notion of the
essential class, a very useful concept in the deterministic case. We prove
that the essential class always exists and is unique, and the Hausdorff
dimension of the attractor can, almost surely, be determined from the
incidence matrices associated with the essential class. When the measure is
`regular', we give a formula for the local dimension of the measure at any
point in terms of Lyapunov exponents of suitable transition matrices, as in
the deterministic case.

Section \ref{sec:5} is devoted to a study of the local dimension theory of a
special class of examples of finite type random IFS and their associated
regular self-similar measures (which do not, in general, satisfy the uniform
strong separation condition). Despite the overlaps, this class has a
`commuting' property that permits us to show that the set of attainable
local dimensions is, again, almost surely an interval for which we give a
formula. The notion of neck levels, discussed in \cite{Jarvenpaa17}, is
useful here, as is Kingman's subadditive ergodic theorem. Particular
examples of such random IFS and their self-similar measures are analyzed in
more detail in Section \ref{sec:6}.

\section{Notation and Definitions}

\label{sec:2}

\subsection{Random iterated function systems}

In this section, we outline the notation that will be used throughout the
remainder of the paper and briefly describe some important properties of
random homogeneous self-similar sets and measures.

Fix $m\in \BbN$. Set 
\begin{equation*}
\mathcal{A}=\{1,\dots,m\}
\end{equation*}
and write $\mathcal{A}^{k}=\{1,\dots ,m\}^{k}$ for all the words from the
alphabet $\mathcal{A}$ of length $k$. Put 
\begin{equation*}
\Omega =\mathcal{A}^{\BbN}=\{1,\dots ,m\}^{\BbN},
\end{equation*}
the set of all infinite sequences with entries in $\mathcal{A}$.

Let $\Prob$ be the Bernoulli probability measure on $\Omega $ giving weight $%
\theta _{j}>0$ to the letter $j$, so $\sum_{j=1}^{m}\theta _{j}=1$. It is
easy to check that $\Prob$ is indeed a measure and for $\omega =(\omega
_{i})_{i=1}^{\infty }\in \Omega $ we have $\Prob\{\omega \in \Omega :\omega
_{k}=j\}=\theta _{j}$ for each $j,k$. Further, $\Prob$ is invariant and
ergodic with respect to the shift map $\pi (\omega )=\pi (\omega _{1},\omega
_{2},\omega _{3},\dots )=(\omega _{2},\omega _{3},\omega _{4},\dots )$.

To each letter $j\in \mathcal{A}$ we associate an iterated function system
(IFS) 
\begin{equation*}
\fS_{j}=\{S_{j,0},S_{j,1},\dots ,S_{j,\cN(j)}\}
\end{equation*}
consisting of finitely many, strictly contracting similitudes on $\BbR^{d}$.
Thus, for each $j,k$ we have contraction factors, $0<r_{j,k}<1$, so that 
\begin{equation*}
\lvert S_{j,k}(x)-S_{j,k}(y)\rvert =r_{j,k}\lvert x-y\rvert .
\end{equation*}
Let $\I_{j}=\{0,1,\dots ,\cN(j)\}$ be the index set for IFS $\fS_{j}$. For
each $\omega \in \Omega $ and integer $n$, let 
\begin{equation*}
\Lambda _{\omega ,n}=\{(\sigma _{1},\dots,\sigma _{n}):\sigma _{j}\in \I%
_{\omega _{j}},j=1,\dots,n\}
\end{equation*}
and $\Lambda _{\omega }=\{(\sigma _{j})_{j=1}^{\infty }:\sigma _{j}\in \I%
_{\omega _{j}}$ for all $j\}$. The elements of $\Lambda _{\omega ,n}$ will
be called codings of length $n$.

Consider the finite family of iterated function systems $\fS=\{\fS%
_{j}:j=1,\dots,m\}$. We call the tuple $(\fS,\Prob)$ (or simply $\fS$) a 
\emph{random homogeneous iterated function system (RIFS)}. Of course, if $%
m=1 $, then the RIFS is simply an IFS.

The finiteness of $\fS$ and $\cN(j)$ ensures that 
\begin{equation*}
r_{\max }=\max_{j,k}r_{j,k}<1\text{ and }r_{\min }=\min_{j,k}r_{j,k}>0.
\end{equation*}

By slight abuse of notation we also consider the $\fS_{j}$ as operators on $%
\mathcal{K}(\BbR^{d})$, the non-empty compact subsets of $\BbR^{d}$, in the
following way. Suppose $K\in \mathcal{K}(\BbR^{d})$ satisfies the property
that $S_{j,k}(K)\subseteq K$. Then define 
\begin{equation*}
\fS_{j}(K)=\bigcup_{k\in \I_{j}}S_{j,k}(K).
\end{equation*}

It is a classical result by Hutchinson~\cite{Hutchinson81} that each IFS $\fS%
_{j}$ has an associated unique invariant compact set $K_{j}$, known as its 
\textit{self-similar set }or \textit{attractor}, satisfying 
\begin{equation*}
\fS_{j}(K_{j})=K_{j}.
\end{equation*}
The uniqueness arises from the fact that $\fS_{j}$ is a contracting map on $(%
\mathcal{K}(\BbR^{d}),d_{H})$, where $d_{H}$ is the Hausdorff metric.
Further, given any sufficiently large $K\in \mathcal{K}(\BbR)$, the set $%
K_{j}$ can be obtained by iterating the map $\fS_{j}$: 
\begin{equation*}
K_{j}=\fS_{j}\circ \fS_{j}\circ \fS_{j}\circ \dots (K)=\bigcap_{k=1}^{\infty
}\underbrace{\fS_{j}\circ \fS_{j}\circ \dots \circ \fS_{j}(K)}_{\text{$k$
times}}.
\end{equation*}

The same idea can be applied to random iterated function systems at every
step in the construction. Indeed, let $(\fS,\Prob)$ be a RIFS as above,
indexed by the letters of $\mathcal{A}$. For each $\omega \in \Omega $ we
define the \emph{random homogeneous self-similar set} or \emph{random
attractor}, $K_{\omega }\in \mathcal{K}(\BbR^{d})$, to be the set obtained
by iterating operators in $\fS$ according to $\omega =(\omega _{1},\omega
_{2},\dots )$,

\begin{equation*}
K_{\omega }=\lim_n \fS_{\omega _{1}}\circ \fS_{\omega _{2}}\circ \dots \circ %
\fS_{\omega _{n}}(K)
\end{equation*}
where the limit is taken in the Hausdorff metric and $K$ is any non-empty
compact subset of $\BbR^{d}$. A suitable application of the contraction
mapping principle shows that this limit exists and is unique. The set $%
K_{\omega}$ also arises as

\begin{equation*}
K_{\omega }=\bigcap_{k=1}^{\infty }\fS_{\omega _{1}}\circ \fS_{\omega
_{2}}\circ \dots \circ \fS_{\omega _{k}}(K),
\end{equation*}
where $K\in \mathcal{K}(\BbR^{d})$ is a sufficiently large compact set.

\begin{example}
\label{randomCantor}Cantor sets with ratios of dissection chosen randomly
from a finite set are examples of such random self-similar sets. Suppose we
are given $0<r_{j}<1/2$, $j=1,\dots,m$, and probability measure $\mathbb{P}$%
. Consider the RIFS with contractions $S_{j,k}(x)=r_{j}x+(1-r_{j})k$, $k=0,1$%
. The self-similar set $K_{\omega }$ is the random Cantor set where if $%
\omega =(\omega _{i})\in (\Omega ,\mathbb{P)}$, then at step $i$ in the
usual Cantor set construction, we remove from each of the parent Cantor
intervals of step $i-1$, the middle open interval, keeping the outer closed
intervals of length $r_{\omega _{1}}\dots r_{\omega _{i}}$. The more general
homogeneous Cantor sets of \cite{F4} could similarly be randomized.
\end{example}

The Hausdorff and box-counting dimensions of this class of examples coincide
almost surely. Further, they coincide almost surely with the unique real
number $s$ satisfying 
\begin{equation}
0=\BbE\left( \log \sum_{k\in \I_{\omega _{1}}}r_{\omega _{1},k}^{s}\right)
=\sum_{j=1}^{m}\theta _{j}\log \sum_{k\in \I_{j}}r_{j,k}^{s},  \notag
\end{equation}
or, equivalently, 
\begin{equation}
\prod_{j=1}^{m}\left( \sum_{k\in \I_{j}}r_{j,k}^{s}\right) ^{\theta _{j}}=1.
\label{Hdim}
\end{equation}
Note that when we write $\mathbb{E}$ we mean the expectation with respect to 
$\mathbb{P}$ unless we specify otherwise.

We remark that, in contrast with the case of a single IFS, the Hausdorff $s$%
-measure of the random attractor $K_{\omega }$ is typically zero for a.a.\ $%
\omega $ when $s$ is the Hausdorff dimension (see \cite%
{Ha,Roy11,Troscheit17b}).

\subsection{Self-similar measures and their local dimensions}

To each letter $j\in \mathcal{A}$, we associate a probability vector on
their respective index set $\I_{j}$. That is, we have probabilities $%
p_{j,k}>0$ such that $\sum_{k\in \I_{j}}p_{j,k}=1$ for each $j$. For each $%
\omega \in \Omega$ there then exists a unique probability measure, supported
on $K_{\omega }$, that satisfies 
\begin{equation}
\mu _{\omega }(E)=\sum_{k\in \I_{\omega _{1}}}p_{\omega _{1},k}\,\mu
_{\omega _{2}\omega _{3}\omega _{4}\dots }\circ {S_{\omega _{1},k}}^{-1}(E)
\label{randommeas}
\end{equation}
for all Borel sets $E$. We refer to $\mu _{\omega }$ as the \emph{random
homogeneous self-similar measure} for $\omega \in \Omega $. As with the
random attractor, the existence and uniqueness of the random self-similar
measure also follows from a suitable application of the contraction mapping
principle.

\begin{example}
\label{Cantormeas}If we take the RIFS from Example \ref{randomCantor}, with
probabilities $p_{j,0}=p$, $p_{j,1}=1-p$ for all $j$, then the associated
random self-similar measure is the $p$-Cantor measure supported on the
random Cantor set $K_{\omega }$.
\end{example}

We choose $\omega $ according to $\Prob$ and aim to study the generic
dimensional properties of these random self-similar measures.

\begin{definition}
Given a probability measure $\mu $, the \textbf{upper local dimension} of $%
\mu $ at $x\in \supp\mu $ is 
\begin{equation*}
\dimulc\mu (x)=\limsup_{\varepsilon \to 0^{+}}\frac{\log \mu (B(x,\epsilon))%
} {\log \varepsilon}.
\end{equation*}
Replacing the $\limsup $ by $\liminf $ gives the \textbf{lower local
dimension}, denoted $\dimllc\mu (x)$. If the limit exists, we call the
number the \textbf{local dimension} of $\mu $ at $x$ and denote this by $%
\dimlc\mu (x)$.
\end{definition}

Given a random homogeneous self-similar measure $\mu _{\omega }$ with
support the self-similar set $K_{\omega }$, let 
\begin{align*}
E_{\omega ,\alpha }& =\{x\in K_{\omega }:\dimlc\mu _{\omega }(x)=\alpha \},
\\
\overline{E_{\omega ,\alpha }}& =\{x\in K_{\omega }:\dimulc\mu _{\omega
}(x)=\alpha \}, \\
\underline{E_{\omega ,\alpha }}& =\{x\in K_{\omega }:\dimllc\mu _{\omega
}(x)=\alpha \}.
\end{align*}
The \emph{multifractal spectrum }of $\mu _{\omega }$ is the function $%
f_{\omega }(\alpha )=\dim _{H}E_{\omega ,\alpha }$.

Henceforth we will omit the adjective `homogeneous', although it will
always be understood.

\section{Multifractal Analysis under Strong Separation}

\label{sec:3}

\subsection{Separation conditions}

In order to obtain meaningful results on the multifractal analysis of random
self-similar measures, we will make use of different separation conditions.
In this section, we will assume the RIFS satisfies the uniform strong
separation condition, a random variant of the strong separation condition
(SSC).

\begin{definition}
Let $\fS=\{S_{1},\dots ,S_{N}\}$ be an IFS of contracting maps with
self-similar set $K$. We say that $\fS$ satisfies the \textbf{strong
separation condition (SSC)} if 
\begin{equation*}
S_{k_{1}}(K)\cap S_{k_{2}}(K)=\emptyset \quad\text{ for all distinct }\quad
k_{1},k_{2}\in \{1,\dots ,N\}.
\end{equation*}

Equivalently, there exists $\epsilon >0$ such that 
\begin{equation*}
\inf_{1\leq k_{1}<k_{2}\leq N}\;\inf_{x,y\in K}\;\lvert
S_{k_{1}}(x)-S_{k_{2}}(y)\rvert \geq \epsilon .
\end{equation*}
\end{definition}

\begin{definition}
Let $\fS$ be a family of iterated function systems of contracting maps, $\{%
\fS_{1},\dots,\fS_{m}\}$. Let $K\in \mathcal{K}(\BbR^{d})$ be the smallest
set (with respect to its diameter) such that $S_{j,k}(K)\subseteq K$ for all 
$S_{j,k}\in \fS_{j}$, $j=1,\dots,m$. We say that $\fS$ satisfies the \textbf{%
uniform strong separation condition (USSC)} if there exists $\epsilon >0$
such that 
\begin{equation*}
\inf_{j}\inf_{\substack{ k_{1},k_{2}\in \I_{j}  \\ k_{1}\neq k_{2}}}
\;\inf_{x,y\in K}\;\lvert S_{j,k_{1}}(x)-S_{j,k_{2}}(y)\rvert \geq \epsilon .
\end{equation*}
\end{definition}

This is also sometimes known as the very strong separation condition.

Clearly, the USSC is satisfied if each $\fS_{j}$, $j=1,\dots,m$, satisfies
the SSC. The random Cantor sets of Example \ref{randomCantor} satisfy the
USSC.

Given $\omega \in \Omega $ and $\sigma =(\sigma _{1},\sigma _{2},\dots
,\sigma _{n})$ with $\sigma _{i}\in \I_{\omega _{i}}$, we let $[\sigma
_{1},\dots ,\sigma _{n}]_{\omega }$ denote the $n$-level $\omega$-cylinder,
that is 
\begin{equation*}
\lbrack \sigma _{1},\dots ,\sigma _{n}]_{\omega }=\{(\tau _{1},\tau
_{2},\dots )\;:\;\tau _{i}\in \mathcal{I}_{\omega _{i}}\text{ for all }i 
\text{, }\tau _{i}=\sigma _{i}\text{ for }1\leq i\leq n\}.
\end{equation*}
Under the assumption of the uniform strong separation condition we can
directly relate these symbolic cylinders to geometric cylinders. Let $K$ be
the compact set arising in the definition of the USSC. By slight abuse of
notation we also consider a cylinder to be the set 
\begin{equation*}
\lbrack \sigma _{1},\dots ,\sigma _{n}]_{\omega }=S_{\omega _{1},\sigma
_{1}}\circ \dots \circ S_{\omega _{n},\sigma _{n}}(K)\cap K_{\omega }.
\end{equation*}
This equivalence follows from the fact that the gaps in the images of $%
S_{j,k}$ do not overlap and so each point $x$ in the attractor $K_{\omega }$
has a unique symbolic encoding $\sigma $. Given such an $x$, we see that $%
x\in \bigcap_{n=1}^{\infty }[\sigma _{1},\dots ,\sigma _{n}]_{\omega }$ and
write 
\begin{equation*}
C(\omega ,x,n)=[\sigma _{1},\dots ,\sigma _{n}]_{\omega }
\end{equation*}
for the unique $n$-level $\omega $-cylinder containing $x$.

Note that 
\begin{equation*}
\mu _{\omega }([\sigma _{1},\dots ,\sigma _{n}]_{\omega })=p_{\omega
_{1},\sigma _{1}}\dots p_{\omega _{n},\sigma _{n}}
\end{equation*}
and the diameter of the cylinder $[\sigma _{1},\dots ,\sigma _{n}]_{\omega }$
is given by 
\begin{equation*}
\diam([\sigma _{1},\dots ,\sigma _{n}]_{\omega })=\,r_{\omega _{1},\sigma
_{1}}\dots r_{\omega _{n},\sigma _{n}}\diam K.
\end{equation*}

\subsection{Multifractal analysis for RIFS satisfying USSC}

We will now consider the multifractal spectrum for random homogeneous
measures satisfying the uniform strong separation condition. We note that we
make no special assumption on the contraction rates other than that they
satisfy $0<r_{j,k}<1$ for all letters $j$ and $k\in \I_{j}$.

In order to state our main result we need to introduce additional notation.
Let 
\begin{equation}
\overline{\alpha }=\max \frac{\log \prod_{j=1}^{m}(p_{j,k_{j}})^{\theta_{j}}%
} {\log \prod_{j=1}^{m}(r_{j,k_{j}})^{\theta_{j}}}\text{ and } \underline{%
\alpha}=\min \frac{\log \prod_{j=1}^{m}(p_{j,k_{j}})^{\theta _{j}}}{\log
\prod_{j=1}^{m}(r_{j,k_j})^{\theta_{j}}} ,  \label{alpha}
\end{equation}
where the maximum (or minimum) is taken over all valid choices of
probabilities $p_{j,k_{j}}$ and contraction factors $r_{j,k_{j}}$.

Given a real number $q$, define $\beta (q)$ by 
\begin{equation}
0=\BbE\left( \log \sum_{k\in \I_{\omega _{1}}}p_{\omega _{1},k}^{q}r_{\omega
_{1},k}^{\beta (q)}\right) =\sum_{j=1}^{m}\theta _{j}\log \sum_{k\in \I%
_{j}}p_{j,k}^{q}r_{j,k}^{\beta (q)},  \notag
\end{equation}
or, equivalently, 
\begin{equation}
\prod_{j=1}^{m}\left( \sum_{k\in \I_{j}}p_{j,k}^{q}r_{j,k}^{\beta
(q)}\right) ^{\theta _{j}}=1.  \label{DefBeta2}
\end{equation}
From (\ref{Hdim}), we see that $\beta (0)$ is the a.s.\ Hausdorff dimension
of $K_{\omega }$.

Differentiating implicitly with respect to $q$ gives 
\begin{equation*}
-\beta ^{\prime }(q)=\frac{\sum_{j=1}^{m}\theta _{j}\sum_{k\in \I%
_{j}}p_{j,k}^{q}r_{j,k}^{\beta (q)}\log p_{j,k}/D_{j}(q)}{%
\sum_{j=1}^{m}\theta _{j}\sum_{k\in \I_{j}}p_{j,k}^{q}r_{j,k}^{\beta
(q)}\log r_{j,k}/D_{j}(q)},
\end{equation*}
where 
\begin{equation}
D_{j}(q)=\sum_{k\in \I_{j}}p_{j,k}^{q}r_{j,k}^{\beta (q)}.  \label{Dq}
\end{equation}

Here is our main result of this section.

\begin{theorem}
\label{thm:main} Let $(\fS,\Prob)$ be a random iterated function system that
satisfies the uniform strong separation condition. Then, for $\Prob$-a.a.\ $%
\omega \in \Omega $, 
\begin{align*}
\lbrack \underline{\alpha },\overline{\alpha }]& =\{\dimlc\mu _{\omega
}(x)\;:\;\;x\in K_{\omega }\} \\
& =\{\dimllc\mu _{\omega }(x)\;:\;\;x\in K_{\omega }\} \\
& =\{\dimulc\mu _{\omega }(x)\;:\;\;x\in K_{\omega }\}.
\end{align*}
If $\alpha \in (\underline{\alpha },\overline{\alpha })$, then there is some 
$q_{\alpha }\in \mathbb{R}$ such that $\alpha =-\beta ^{\prime }(q_{\alpha
}) $ and for a.a.\ $\omega $, 
\begin{align*}
\dim _{H}E_{\omega ,\alpha }=& \inf_{q}(\beta (q)+\alpha q)=\beta (q_{\alpha
})-\beta ^{\prime }(q_{\alpha })q_{\alpha } \\
=& \dim _{H}\underline{E_{\omega ,\alpha }}=\dim _{H}\overline{E_{\omega
,\alpha }}.
\end{align*}
Moreover, the function $\beta $ is convex.
\end{theorem}

\begin{example}
Take the random Cantor set and $p$-Cantor measures of Example \ref%
{Cantormeas} with $p\leq 1-p$. Then a.s.\ the set of attainable local
dimensions is the closed interval 
\begin{equation*}
\left[ \frac{\log (1-p)}{\log \prod_{j=1}^{m}r_{j}^{\theta _{j}}}, \frac{%
\log (p)}{\log \prod_{j=1}^{m}r_{j}^{\theta _{j}}}\right] .
\end{equation*}
See \cite{HY} for similar results for the deterministic Cantor measures on
Cantor sets with variable ratios of dissection.
\end{example}

Our proof will closely follow the strategy given in Falconer \cite[Chapter 11%
]{FalcTech} and \cite[Chapter 17]{FractalGeo3} for the deterministic strong
separation case. \vskip0.5em\noindent We begin by proving a number of
technical results. First, we note that to determine local dimensions, we can
work with cylinders rather than balls. This is standard, so the proof is
omitted.

\begin{lemma}
\label{locdim}For all $x\in K_{\omega }$, $\dimlc\mu _{\omega }(x)=\alpha $
if and only if 
\begin{equation*}
\lim_{n\rightarrow \infty }\frac{\log \mu _{\omega }(C(\omega ,x,n))}{\log
\left\vert C(\omega ,x,n)\right\vert }=\alpha .
\end{equation*}%
Similar statements hold for upper/lower local dimensions.
\end{lemma}

Similar to \cite{FalcTech}, for each $\omega \in \Omega $ and real number $q$
we define an auxiliary random probability measure $\nu _{\omega ,q}$ by 
\begin{equation*}
\nu _{\omega ,q}([\sigma _{1},\dots ,\sigma _{n}]_{\omega })=\frac{%
\prod_{i=1}^{n}p_{\omega _{i},\sigma _{i}}^{q}r_{\omega _{i},\sigma
_{i}}^{\beta (q)}}{\prod_{i=1}^{n}D_{\omega _{i}}(q)}
\end{equation*}%
where $D_{\omega _{i}}(q)$ is defined in (\ref{Dq}). We leave the
verification that this defines a measure to the reader.

This measure is useful because 
\begin{align*}
\frac{\log \nu _{\omega ,q}([\sigma _{1},\dots ,\sigma _{n}])}{\log
\left\vert [\sigma _{1},\dots ,\sigma _{n}]\right\vert }& =\frac{q\log
\prod_{i=1}^{n}p_{\omega _{i},\sigma _{i}}+\beta (q)\log
\prod_{i=1}^{n}r_{\omega _{i},\sigma _{i}}-\log \prod_{i=1}^{n}D_{\omega
_{i}}(q)}{\log \prod_{i=1}^{n}r_{\omega _{i},\sigma _{i}}} \\
& =\frac{q\log \mu _{\omega }([\sigma _{1},\dots ,\sigma _{n}])}{\log
\left\vert [\sigma _{1},\dots ,\sigma _{n}]\right\vert }+\beta (q)-\frac{%
\log \prod_{i=1}^{n}D_{\omega _{i}}(q)}{\log \prod_{i=1}^{n}r_{\omega
_{i},\sigma _{i}}}.
\end{align*}

Our interest is in studying the behaviour of this expression as $%
n\rightarrow \infty $. Observe that Birkhoff's ergodic theorem (B.E.T.) and
the definition of $\beta (q)$ implies that for a.a.\ $\omega $, 
\begin{equation*}
\lim_{n\rightarrow \infty }\frac{1}{n}\log \prod_{i=1}^{n}D_{\omega
_{i}}(q)=\lim_{n\rightarrow \infty }\frac{1}{n}\sum_{i=1}^{n}\log
\sum_{\sigma _{i}\in \I_{\omega _{i}}}p_{\omega _{i},\sigma
_{i}}^{q}r_{\omega _{i},\sigma _{i}}^{\beta (q)}=\BbE\left( \log
\sum_{\sigma _{i}\in \I_{\omega _{i}}}p_{\omega _{i},\sigma
_{i}}^{q}r_{\omega _{i},\sigma _{i}}^{\beta (q)}\right) =0.
\end{equation*}

Since it is also the case that $\left\vert \log \prod_{i=1}^{n}r_{\omega
_{i},\sigma _{i}}\right\vert \leq n\left\vert \log r_{\min }\right\vert $,
it follows that 
\begin{equation*}
\lim_{n}\frac{\log \prod_{i=1}^{n}D_{\omega _{i}}(q)}{\log
\prod_{i=1}^{n}r_{\omega _{i},\sigma _{i}}}=0\text{.}
\end{equation*}

These comments show that

\begin{lemma}
For a.a.\ $\omega $, 
\begin{equation*}
\lim_{n\rightarrow \infty }\frac{\log \nu _{\omega ,q}([\sigma _{1},\dots
,\sigma _{n}])}{\log \left\vert [\sigma _{1},\dots ,\sigma _{n}]\right\vert }
=q\lim_{n}\frac{\log \mu _{\omega }([\sigma _{1},\dots ,\sigma _{n}])}{\log
\left\vert [\sigma _{1},\dots ,\sigma _{n}]\right\vert }+\beta (q)
\end{equation*}
and similarly for $\limsup $ and $\liminf $.
\end{lemma}

A similar statement to Lemma \ref{locdim} also holds for $\nu _{\omega
,\beta}$ and this proves

\begin{corollary}
\label{vLocDim}For a.a.\ $\omega $ and every real number $q$, 
\begin{equation*}
\dimlc\nu _{\omega ,q}(x)=q\dimlc\mu _{\omega }(x)+\beta (q)\text{ for }x\in
K_{\omega }.
\end{equation*}
Analogous statements hold for the upper and lower local dimensions.
\end{corollary}

The next step is to show that $\nu _{\omega ,q}$ is concentrated on $%
E_{\omega ,\alpha }$ for a.a. $\omega $, where $\alpha =-\beta ^{\prime }(q)$%
. We will then apply the \emph{mass distribution principle} to determine $%
\dim _{H}E_{\omega ,\alpha }$.

\begin{proposition}
\label{AuxMeas}Let $q\in \BbR$ and $\alpha =-\beta ^{\prime }(q)$. Then $\nu
_{\omega ,q}(E_{\omega ,\alpha })=1$ for a.a.~$\omega $.
\end{proposition}

\begin{proof}
Fix $\varepsilon ,\delta >0$. As in \cite[Proposition 11.4]{FalcTech},%
\begin{equation*}
\nu _{\omega ,q}\{x:\mu _{\omega }(C(\omega ,n,x))\geq \left\vert C(\omega
,n,x)\right\vert ^{\alpha -\varepsilon }\}=\frac{\prod_{i=1}^{n}\sum_{\sigma
_{i}\in \I_{\omega _{i}}}(p_{\omega _{i},\sigma _{i}})^{\delta +q}(r_{\omega
_{i},\sigma _{i}})^{\delta (\varepsilon -\alpha )+\beta (q)}}{%
\prod_{i=1}^{n}\sum_{\sigma _{i}\in \I_{\omega _{i}}}(p_{\omega _{i},\sigma
_{i}})^{q}(r_{\omega _{i},\sigma _{i}})^{\beta (q)}}.
\end{equation*}

For $q,\zeta \in \BbR$, put 
\begin{equation*}
\Phi (q,\zeta )=\sum_{j=1}^{m}\theta _{j}\log \sum_{k\in \I%
_{j}}p_{j,k}^{q}r_{j,k}^{\zeta }.
\end{equation*}%
Choose sequences $(\varepsilon _{n}),(\delta _{n})$ tending to $0$. By the
B.E.T., for each $n,\ell $ 
\begin{equation*}
\Phi (q+\delta _{n},\beta (q)+(\varepsilon _{\ell }-\alpha )\delta
_{n})=\lim_{N\rightarrow \infty }\frac{1}{N}\log \prod_{i=1}^{N}\sum_{\sigma
_{i}\in \I_{\omega _{i}}}(p_{\omega _{i},\sigma _{i}})^{\delta
_{n}+q}(r_{\omega _{i},\sigma _{i}})^{\delta _{n}(\varepsilon _{\ell
}-\alpha )+\beta (q)}
\end{equation*}%
for a.a.\ $\omega $. Let $\Omega _{0}$ be the intersection of these
countably many sets of full measure, together with the full measure set, $%
\{\omega :\lim_{N}\frac{1}{N}\prod_{i=1}^{N}\log D_{\omega _{i}}(q)=0\}$.
Then $\Omega _{0}$ is also of full measure.

Since $\alpha =-\beta ^{\prime }(q)$ and $\Phi $ is strictly decreasing in
the second variable, a Taylor series argument shows that for $\delta =\delta
(\varepsilon )>0$ sufficiently small we have

\begin{align}
\Phi (q+\delta ,\beta (q)+(\varepsilon -\alpha )\delta )& <0  \label{Claim1}
\\
\Phi (q-\delta ,\beta (q)+(\varepsilon +\alpha )\delta )& <0.  \label{Claim}
\end{align}

Given $\varepsilon =\varepsilon _{\ell }$, choose $\delta _{n}=\delta $
sufficiently small so that the relation (\ref{Claim1}) applies. Put $A=-\Phi
(q+\delta ,\beta (q)+(\varepsilon -\alpha )\delta )/2>0$. For $\omega \in
\Omega _{0}$ and large enough $n$, say $n\geq n_{\omega }$, 
\begin{equation*}
\log \prod_{i=1}^{n}\sum_{\sigma _{i}\in \I_{\omega _{i}}}(p_{\omega
_{i},\sigma _{i}})^{\delta _{n}+q}(r_{\omega _{i},\sigma _{i}})^{\delta
_{n}(\varepsilon -\alpha )+\beta (q)}\leq -nA
\end{equation*}%
and 
\begin{equation*}
\log \prod_{i=1}^{n}\sum_{\sigma _{i}\in \I_{\omega _{i}}}p_{\omega
_{i},\sigma _{i}}^{q}r_{\omega _{i},\sigma _{i}}^{\beta (q)}\leq nA/2.
\end{equation*}%
Thus, for all $\omega \in \Omega _{0}$ we have 
\begin{equation*}
\nu _{\omega ,q}\{x:\mu _{\omega }(C(\omega ,n,x))\geq \left\vert C(\omega
,n,x)\right\vert ^{\alpha -\varepsilon }\}\leq \exp (-nA/2)\text{ for each }%
n\geq n_{\omega }.
\end{equation*}%
Therefore 
\begin{equation*}
\nu _{\omega ,q}\{x:\mu (C(\omega ,n,x))\geq \left\vert C(\omega
,n,x)\right\vert ^{\alpha -\varepsilon }\text{ for some }n\geq n_{\omega
}\}\leq \sum_{n=n_{\omega }}^{\infty }\exp (-nA/2)
\end{equation*}%
and this tends to zero as $n_{\omega }\rightarrow \infty $. Hence, for all $%
\omega \in \Omega _{0}$, 
\begin{equation*}
\liminf_{n\rightarrow \infty }\frac{\log \mu _{\omega }(C(\omega ,n,x))}{%
\log \left\vert C(\omega ,n,x)\right\vert }\geq \alpha -\varepsilon \text{
for }\nu _{\omega ,q}\text{ a.a. }x\text{.}
\end{equation*}

We similarly deduce that 
\begin{equation*}
\limsup_{n\rightarrow \infty }\frac{\log \mu _{\omega }(C(\omega ,n,x))}{%
\log \left\vert C(\omega ,n,x)\right\vert }\leq \alpha +\varepsilon \text{
for }\nu _{\omega ,q}\text{ a.a. }x,
\end{equation*}
using the inequality given in (\ref{Claim}). As this holds for all $%
\varepsilon =\varepsilon _{\ell }>0$, it follows that 
\begin{equation*}
\lim_{n\rightarrow \infty }\frac{\log \mu _{\omega }(C(\omega ,n,x))}{\log
\left\vert C(\omega ,n,x)\right\vert }=\alpha \text{ for }\nu _{\omega ,q}%
\text{ a.a. }x.
\end{equation*}
Consequently, for all $\omega \in \Omega _{0}$ we have $\nu _{\omega
,q}(E_{\omega ,\alpha })=1$.
\end{proof}

\begin{corollary}
\label{DimEalpha}For a.a.\ $\omega $ and for $\alpha =-\beta ^{\prime }(q)$,
we have 
\begin{equation*}
\dim _{H}E_{\omega ,\alpha }=q\alpha +\beta (q)=\dim _{H}\overline{E_{\omega
,\alpha }}=\dim _{H}\underline{E_{\omega ,\alpha }}.
\end{equation*}
\end{corollary}

\begin{proof}
According to Corollary \ref{vLocDim}, $\dimlc\nu _{\omega ,q}(x)=q\alpha
+\beta (q)$ for every $x\in E_{\omega ,\alpha }$ and a.a.\ $\omega \in
\Omega $. Furthermore, $\nu _{\omega ,q}$ is concentrated on $E_{\omega
,\alpha }$. Thus the mass distribution principle implies $\dim _{H}E_{\omega
,\alpha }=q\alpha +\beta (q)$ for a.a.\ $\omega $.

Similarly, $\dimllc\nu _{\omega ,q}(x)=q\alpha +\beta (q)$ for every $x\in 
\underline{E_{\omega ,\alpha }}$ and as $\underline{E_{\omega ,\alpha }}%
\supseteq E_{\omega ,\alpha }$, the measure $\nu _{\omega ,q}$ is also
concentrated on \underline{$E_{\omega ,\alpha }$}. Hence $\dim _{H}%
\underline{E_{\omega ,\alpha }}=q\alpha +\beta (q)$ for a.a.\ $\omega $ and
similarly for the upper local dimension.
\end{proof}

Recall that $\underline{\alpha }$ and $\overline{\alpha }$ were defined in (%
\ref{alpha}). Any vector $(\sigma _{1},\dots ,\sigma _{m})$ with $\sigma
_{j}\in \I_{j}$ for $j=1,2,\dots ,m$ satisfying 
\begin{equation*}
\frac{\log \prod_{j=1}^{m}p_{j,\sigma _{j}}^{\theta _{j}}}{\log
\prod_{j=1}^{m}r_{j,\sigma _{j}}^{\theta _{j}}}=\text{ }\underline{\alpha }%
\qquad \text{ (or }\overline{\alpha })
\end{equation*}
will be called a \textit{minimizing} (resp., \textit{maximizing}) vector.

We will use the following properties of minimizing/maximizing vectors.

\begin{lemma}
\label{L:AllMin}Assume $(\tau_{1},\dots,\tau_{m})$ and $(\sigma_{1},\dots,%
\sigma_{m})$ are both minimizing (or maximizing) vectors. Then for any index 
$k$, so is $(\tau_{1},\dots,\tau_{k-1},\sigma_{k},\tau_{k+1},\dots,\tau_{m})$%
.
\end{lemma}

\begin{proof}
The definition of $\underline{\alpha }$ means that for all choices of $%
(\gamma _{1},\dots ,\gamma _{m})$, with $\gamma _{j}\in \I_{j}$, we have 
\begin{equation*}
\sum_{j=1}^{m}\log p_{j,\gamma _{j}}^{\theta _{j}}-\underline{\alpha }%
\sum_{j=1}^{m}\log r_{j,\gamma _{j}}\leq 0,
\end{equation*}
with equality if $(\gamma _{1},\dots ,\gamma _{m})$ is a minimizing vector
and strict inequality if it is not.

Without loss of generality assume the index $k$ of the statement of the
lemma is $k=1$. As $(\tau _{1},\dots ,\tau _{m})$ and $(\sigma _{1},\dots
,\sigma _{m})$ are minimizing vectors, 
\begin{equation*}
0=\sum_{j=1}^{m}\log p_{j,\tau _{j}}^{\theta _{j}}+\sum_{j=1}^{m}\log
p_{j,\sigma _{j}}^{\theta _{j}}-\underline{\alpha }\left( \sum_{j=1}^{m}\log
r_{j,\tau _{j}}^{\theta _{j}}+\sum_{j=1}^{m}\log r_{j,\sigma _{j}}^{\theta
_{j}}\right) .
\end{equation*}
Rearranging terms, we have 
\begin{multline*}
0=\left[ \log p_{1,\sigma _{1}}^{\theta _{1}}+\sum_{j=2}^{m}\log p_{j,\tau
_{j}}^{\theta _{j}}-\underline{\alpha }(\log r_{1,\sigma _{1}}^{\theta
_{1}}+\sum_{j=2}^{m}\log r_{j,\tau _{j}}^{\theta _{j}})\right] \\
+\left[ \log p_{1,\tau _{1}}^{\theta _{1}}+\sum_{j=2}^{m}\log p_{j,\sigma
_{j}}^{\theta _{j}}-\underline{\alpha }(\log r_{1,\tau _{1}}^{\theta
_{1}}+\sum_{j=2}^{m}\log r_{j,\sigma _{j}}^{\theta _{j}})\right] .
\end{multline*}
Since both square-bracketed terms are non-positive, both must equal $0$.
Hence $(\sigma _{1},\tau _{2},\dots ,\tau _{m})$ and $(\tau _{1},\sigma
_{2},\dots ,\sigma _{m})$ are both minimizing vectors.
\end{proof}

\begin{lemma}
\label{L:min}Let $i\in \{1,\dots ,m\}$ be given and assume the index $\sigma
_{i}$ has the property that there is no choice of indices $\sigma _{1},\dots
,\sigma _{i-1},\sigma _{i+1},\dots ,\sigma _{m}$ such that $(\sigma
_{1},\dots ,\sigma _{m})$ is a minimizing vector. Then 
\begin{equation*}
\frac{p_{i,\sigma _{i}}^{q}r_{i,\sigma _{i}}^{\beta (q)}}{D_{i}(q)}%
\rightarrow 0\text{ as }q\rightarrow \pm \infty \text{.}
\end{equation*}
\end{lemma}

\begin{proof}
Without loss of generality $i=1=\sigma_{1}$. The assumption that no vector $%
(1,\sigma_{2},\dots,\sigma_{m})$ is minimizing, and the fact that there are
only finitely many probabilities and contraction factors, ensures that there
is some $\delta >0$ such that 
\begin{equation}
\frac{\log \prod_{j=1}^{m}p_{j,\sigma_{j}}^{\theta _{j}}}{\log
\prod_{j=1}^{m}r_{j,\sigma_{j}}^{\theta _{j}}}\geq \text{ }\underline{\alpha 
}+\delta  \label{min}
\end{equation}
for all choices of $\sigma_{j}\in \I_{j}$ with $j=2,\dots,m$.

We proceed by contradiction and assume there is a sequence $q_{n}\rightarrow
\infty $ and $\varepsilon >0$ such that 
\begin{equation*}
p_{1,1}^{q_{n}}r_{1,1}^{\beta (q_{n})}\geq \varepsilon
D_{1}(q_{n})=\varepsilon \sum_{k\in \I_{1}}p_{1,k}^{q_{n}}r_{1,k}^{\beta
(q_{n})}\text{ for all }n\text{.}
\end{equation*}
Coupled with the definition of $\beta (q)$ (see \ref{DefBeta2}) this gives 
\begin{equation*}
\prod_{j=2}^{m}\left( \sum_{k\in \I_{j}}p_{j,k}^{q_{n}}r_{j,k}^{\beta
(q_{n})}\right) ^{\theta _{j}}\left( p_{1,1}^{q_{n}}r_{1,1}^{\beta
(q_{n})}\right) ^{\theta _{1}}\geq \varepsilon ^{\theta _{1}}\text{.}
\end{equation*}

Assume that $k_{j}$ is the index such that $p_{j,k_{j}}^{q_{n}}r_{j,k_{j}}^{%
\beta (q_{n})}$ is the maximal term in the sum over $\I_{j}$. As $\sum_{k\in %
\I_{j}}p_{j,k}^{q_{n}}r_{j,k}^{\beta (q_{n})}\leq (\cN%
(j)+1)p_{j,k_{j}}^{q_{n}}r_{j,k_{j}}^{\beta (q_{n})}$,  we have 
\begin{equation*}
\prod_{j=2}^{m}\left( p_{j,k_{j}}^{q_{n}}r_{j,k_{j}}^{\beta (q_{n})}\right)
^{\theta _{j}}\left( p_{1,1}^{q_{n}}r_{1,1}^{\beta (q_{n})}\right) ^{\theta
_{1}}\geq \varepsilon _{0}
\end{equation*}%
for $\varepsilon _{0}=\min_{j}\varepsilon ^{\theta _{j}}/\max_{j}(\cN(j)~+~1)
$. Reorganizing gives 
\begin{equation*}
\left( p_{1,1}^{\theta _{1}}p_{2,k_{2}}^{\theta _{2}}\dots
p_{m,k_{m}}^{\theta _{m}}\right) ^{q_{n}}\left( r_{1,1}^{\theta
_{1}}r_{2,k_{2}}^{\theta _{2}}\dots r_{m,k_{m}}^{\theta _{m}}\right) ^{\beta
(q_{n})}\geq \varepsilon _{0}.
\end{equation*}%
Using (\ref{min}) we deduce that 
\begin{equation*}
\left( r_{1,1}^{\theta _{1}}r_{2,k_{2}}^{\theta _{2}}\dots
r_{m,k_{m}}^{\theta _{m}}\right) ^{\beta (q_{n})+(\underline{\alpha }+\delta
)q_{n}}\geq \varepsilon _{0}
\end{equation*}%
and thus 
\begin{equation*}
\underline{\alpha }q_{n}+\beta (q_{n})\leq \frac{\log \varepsilon _{0}}{\log
r_{1,1}^{\theta _{1}}r_{2,k_{2}}^{\theta _{2}}\dots r_{m,k_{m}}^{\theta _{m}}%
}-\delta q_{n}\leq \frac{\log \varepsilon _{0}}{\log r_{\max }}-\delta
q_{n}=C_{1}-\delta q_{n}
\end{equation*}%
for a suitable constant $C_{1}$. In particular, $\underline{\alpha }%
q_{n}+\beta (q_{n})\rightarrow -\infty $ as $n\rightarrow \infty $.

Now let $(\sigma_{1},\dots,\sigma_{m})$ be any minimizing sequence. As $%
\sum_{k\in \I_{j}}p_{j,k}^{q_{n}}r_{j,k}^{\beta (q_{n})}$ dominates any one
term in the sum, 
\begin{equation*}
\prod_{j=1}^{m}\left( \sum_{k\in \I_{j}}p_{j,k}^{q_{n}}r_{j,k}^{\beta
(q_{n})}\right) ^{\theta _{j}}\geq \prod_{j=1}^{m}\left(
p_{j,\sigma_{j}}^{q_{n}}r_{j,\sigma_{j}}^{\beta (q_{n})}\right) ^{\theta
_{j}}=\prod_{j=1}^{m}\left( r_{j,\sigma_{j}}^{\beta (q_{n})+\underline{%
\alpha }q_{n}}\right) ^{\theta _{j}}.
\end{equation*}
Since $\underline{\alpha }q_{n}+\beta (q_{n})\leq C_{1}-\delta q_{n}<0$ (for
large $n)$ it follows that for a new constant $C_{2}>0$, 
\begin{equation*}
\prod_{j=1}^{m}\left( \sum_{k\in \I_{j}}p_{j,k}^{q_{n}}r_{j,k}^{\beta
(q_{n})}\right) ^{\theta _{j}}\geq C_{2}r_{\max }^{-\delta q_{n}}\rightarrow
\infty\quad \text{ as }\quad n\rightarrow \infty.
\end{equation*}
This contradicts the definition of $\beta (q)$, which proves the result.
\end{proof}

\begin{lemma}
\label{RangeBeta}The function $\beta $ satisfies $-\beta ^{\prime
}(q)\rightarrow \underline{\alpha }$ as $q\rightarrow \infty $ and $-\beta
^{\prime }(q)\rightarrow \overline{\alpha }$ as $q\rightarrow -\infty $.
\end{lemma}

\begin{proof}
Recall that 
\begin{equation*}
-\beta ^{\prime }(q)=\frac{\sum_{j=1}^{m}\theta _{j}\sum_{k\in \I%
_{j}}p_{j,k}^{q}r_{j,k}^{\beta (q)}\log p_{j,k}/D_{j}(q)}{%
\sum_{j=1}^{m}\theta _{j}\sum_{k\in \I_{j}}p_{j,k}^{q}r_{j,k}^{\beta
(q)}\log r_{j,k}/D_{j}(q)}.
\end{equation*}
For the duration of this lemma, we will let 
\begin{equation*}
b_{j,k}(q)=\frac{p_{j,k}^{q}r_{j,k}^{\beta (q)}}{D_{j}(q)}.
\end{equation*}

Let $\J_{j}$ be the set of indices $\sigma _{j}\in \I_{j}$ such that there
is some choice of $\tau _{i}\in \I_{i}$, for each $i\neq j$, so that the
vector $(\tau _{1},\dots ,\tau _{j-1},\sigma _{j},\tau _{j+1},\dots ,\tau
_{m})$ is minimizing. Lemma \ref{L:AllMin} implies that if $\sigma _{j}\in \J%
_{j}$ for all $j$, then $(\sigma _{1},\dots ,\sigma _{m})$ is a minimizing
choice and hence 
\begin{equation*}
\log p_{1,\sigma _{1}}^{\theta _{1}}\dots p_{m,\sigma _{m}}^{\theta _{m}}=%
\underline{\alpha }\log r_{1,\sigma _{1}}^{\theta _{1}}\dots r_{m,\sigma
_{m}}^{\theta _{m}}.
\end{equation*}%
This means 
\begin{equation*}
\underline{\alpha }=\frac{\sum_{i=1}^{m}\sum_{\sigma _{i}\in \J%
_{i}}\prod_{j=1}^{m}b_{j,\sigma _{j}}(q)\log p_{1,\sigma _{1}}^{\theta
_{1}}\dots p_{m,\sigma _{m}}^{\theta _{m}}}{\sum_{i=1}^{m}\sum_{\sigma
_{i}\in \J_{i}}\prod_{j=1}^{m}b_{j,\sigma _{j}}(q)\log r_{1,\sigma
_{1}}^{\theta _{1}}\dots r_{m,\sigma _{m}}^{\theta _{m}}}.
\end{equation*}

Now, 
\begin{align*}
\sum_{i=1}^{m}\sum_{\sigma _{i}\in \J_{i}}\prod_{j=1}^{m}b_{j,\sigma
_{j}}(q)\log p_{1,\sigma _{1}}^{\theta _{1}}\dots p_{m,\sigma _{m}}^{\theta
_{m}}& =\sum_{i=1}^{m}\sum_{\sigma _{i}\in \J_{i}}b_{1,\sigma _{1}}\dots
b_{m,\sigma _{m}}\sum_{t=1}^{m}\log p_{t,\sigma _{t}}^{\theta _{t}} \\
& =\sum_{t=1}^{m}\sum_{\sigma _{t}\in \J_{t}}b_{t,\sigma _{t}}\log
p_{t,\sigma _{t}}^{\theta _{t}}\prod_{i\neq t}\sum_{\sigma _{i}\in \J%
_{i}}b_{i,\sigma _{i}}
\end{align*}
and similarly for $\sum_{i=1}^{m}\sum_{\sigma _{i}\in \J_{i}}%
\prod_{j=1}^{m}b_{j,\sigma _{j}}(q)\log r_{1,\sigma _{1}}^{\theta _{1}}\dots
r_{m,\sigma _{m}}^{\theta _{m}}$. Hence for every $q$, 
\begin{equation*}
\underline{\alpha }=\frac{\sum_{t=1}^{m}\sum_{\sigma _{t}\in \J%
_{t}}b_{t,\sigma _{t}}\log p_{t,\sigma _{t}}^{\theta _{t}}\prod_{i\neq
t}\sum_{\sigma _{i}\in \J_{i}}b_{i,\sigma _{i}}(q)}{\sum_{t=1}^{m}\sum_{%
\sigma _{t}\in \J_{t}}b_{t,\sigma _{t}}\log r_{t,\sigma _{t}}^{\theta
_{t}}\prod_{i\neq t}\sum_{\sigma _{i}\in \J_{i}}b_{i,\sigma _{i}}(q)}.
\end{equation*}

By Lemma \ref{L:min}, if $\sigma _{i}\notin \J_{i}$ then $b_{i,\sigma
_{i}}(q)\rightarrow 0$ as $q\rightarrow \infty $, so since $\sum_{k\in \I%
_{i}}b_{i,k}(q)=1$, it must be that $\sum_{\sigma _{i}\in \J_{i}}b_{i,\sigma
_{i}}(q)\rightarrow 1$. Therefore 
\begin{equation*}
\underline{\alpha }=\lim_{q\rightarrow \infty }\frac{\sum_{t=1}^{m}\sum_{%
\sigma _{t}\in \J_{t}}b_{t,\sigma _{t}}(q)\log p_{t,\sigma _{t}}^{\theta
_{t}}}{\sum_{t=1}^{m}\sum_{\sigma _{t}\in \J_{t}}b_{t,\sigma _{t}}(q)\log
r_{t,\sigma _{t}}^{\theta _{t}}}\text{.}
\end{equation*}
In terms of this notation, we can write 
\begin{equation*}
-\beta ^{\prime }(q)=\frac{\sum_{i=1}^{m}\left( \sum_{\sigma _{i}\in \J%
_{i}}b_{i,\sigma _{i}}(q)\log p_{i,\sigma _{i}}^{\theta _{i}}+\sum_{\sigma
_{i}\in \I_{i}\backslash \J_{i}}b_{i,\sigma _{i}}(q)\log p_{i,\sigma
_{i}}^{\theta _{i}}\right) }{\sum_{i=1}^{m}\left( \sum_{\sigma _{i}\in \J%
_{i}}b_{i,\sigma _{i}}(q)\log r_{i,\sigma _{i}}^{\theta _{i}}+\sum_{\sigma
_{i}\in \I_{i}\backslash \J_{i}}b_{i,\sigma _{i}}(q)\log r_{i,\sigma
_{i}}^{\theta _{i}}\right) }.
\end{equation*}
Using again the fact that if $\sigma _{i}\notin \J_{i}$ then $b_{i,\sigma
_{i}}(q)\rightarrow 0$ as $q\rightarrow \infty$, it follows that 
\begin{equation*}
\lim_{q\rightarrow \infty }-\beta ^{\prime }(q)=\lim_{q\rightarrow \infty } 
\frac{\sum_{i=1}^{m}\sum_{\sigma _{i}\in \J_{i}}b_{i,\sigma _{i}}(q)\log
p_{i,\sigma _{i}}^{\theta _{i}}}{\sum_{i=1}^{m}\sum_{\sigma _{i}\in \J%
_{i}}b_{i,\sigma _{i}}(q)\log r_{i,\sigma _{i}}^{\theta _{i}}}=\underline{%
\alpha }.
\end{equation*}

The arguments are similar for $\lim_{q\rightarrow -\infty }-\beta ^{\prime
}(q)$.
\end{proof}

\begin{proof}[Proof of Theorem \protect\ref{thm:main}]
First, we will check that the attainable (upper/lower) local dimensions lie
in the interval $[\underline{\alpha },\overline{\alpha }]$ for a.a.\ $\omega$%
.

For each $\omega \in \Omega $, $j=1,\dots ,m$ and $n\in \BbN$, let 
\begin{equation*}
s_{j}(\omega ,n)=\card\{i\in \{1,\dots ,n\}:\omega _{i}=j\}\text{.}
\end{equation*}
By the strong law of large numbers, $s_{j}(\omega ,n)/n\rightarrow \theta
_{j}$ a.s.\ for each $j=1,\dots ,m$, say for all $\omega \in \Omega _{1}$, a
set of full measure. We will actually prove that the set of lower local
dimensions of $\mu _{\omega }$ lie in $[\underline{\alpha },\overline{\alpha 
}]$ for all $\omega \in \Omega _{1}$. The arguments are similar for the
other cases.

Fix $\omega \in \Omega _{1}$ and choose $x\in K_{\omega }$, say $x\in
\bigcap_{n}[\sigma _{1},\dots ,\sigma _{n}]_{\omega }$. Put 
\begin{equation*}
s_{j,k}(x,\omega ,n)=\card\{i\in \{1,\dots ,n\}:(\omega _{i},\sigma
_{i})=(j,k)\},
\end{equation*}
and set $a_{j,k}(x,\omega ,n)=s_{j,k}(x,\omega ,n)/(n\theta _{j})$. Then $%
a_{j,k}\geq 0$ and 
\begin{equation*}
\sum_{k\in \I_{j}}a_{j,k}(x,\omega ,n)=\frac{1}{n\theta _{j}}\sum_{k\in \I%
_{j}}s_{j,k}(x,\omega ,n)=\frac{s_{j}(\omega ,n)}{n\theta _{j}}\rightarrow
1\ \text{ as }\ n\rightarrow \infty
\end{equation*}
since $\omega \in \Omega _{1}$.

Recall from Lemma \ref{locdim} that 
\begin{equation*}
\dimllc\mu _{\omega }(x)=\liminf_{n\rightarrow \infty }\frac{\log \mu
_{\omega }([\sigma _{1},\dots ,\sigma _{n}]_{\omega })}{\log \left\vert
[\sigma _{1},\dots ,\sigma _{n}]_{\omega }\right\vert }=\liminf_{n%
\rightarrow \infty }\frac{\log p_{\omega _{1},\sigma _{1}}\dots p_{\omega
_{n},\sigma _{n}}}{\log r_{\omega _{1},\sigma _{1}}\dots r_{\omega
_{n},\sigma _{n}}}.
\end{equation*}
Thus 
\begin{align}
\dimllc\mu _{\omega }(x)& =\liminf_{n}\frac{\log \prod_{j=1}^{m}\prod_{k\in %
\I_{j}}p_{j,k}^{s_{j,k}(x,\omega ,n)}}{\log \prod_{j=1}^{m}\prod_{k\in \I%
_{j}}r_{j,k}^{s_{j,k}(x,\omega ,n)}}  \notag \\
& =\liminf_{n}\frac{\log \prod_{j=1}^{m}\left( \prod_{k\in \I%
_{j}}p_{j,k}^{a_{j,k}(x,\omega ,n)}\right) ^{n\theta _{j}}}{\log
\prod_{j=1}^{m}\left( \prod_{k\in \I_{j}}r_{j,k}^{a_{j,k}(x,\omega
,n)}\right) ^{n\theta _{j}}}.  \label{DimForm}
\end{align}
For each $\omega \in \Omega _{1}$ and $x\in K_{\omega }$ choose a
subsequence (not renamed) where the $\liminf_{n}$ is actually the limit. Now
choose a further subsequence (also not renamed) such that $a_{j,k}(x,\omega
,n)\rightarrow A_{j,k}(x,\omega )$ as $n\rightarrow \infty $ for each $%
j=1,\dots ,m$ and $k\in \I_{j}$. Then $\sum_{k}A_{j,k}(x,\omega )=1$ for all 
$j$. Taking the limit of equation \eqref{DimForm} along this subsequence we
have 
\begin{equation}
\dimllc\mu _{\omega }(x)=\frac{\log \prod_{j=1}^{m}\left( \prod_{k\in \I%
_{j}}p_{j,k}^{A_{j,k}(x,\omega )}\right) ^{\theta _{j}}}{\log
\prod_{j=1}^{m}\left( \prod_{k\in \I_{j}}r_{j,k}^{A_{j,k}(x,\omega )}\right)
^{\theta _{j}}}.  \label{LowerBd}
\end{equation}

Showing $\dimllc \mu _{\omega }(x)\geq \underline{\alpha }$ is therefore
equivalent to proving 
\begin{equation*}
\log \prod_{j=1}^{m}\left(\prod_{k\in\I_j} p_{j,k}^{A_{j,k}}\right) ^{\theta
_{j}}\leq \underline{\alpha }\log \prod_{j=1}^{m}\left(\prod_{k\in\I_j}
r_{j,k}^{A_{j,k}}\right)^{\theta_{j}},
\end{equation*}
or 
\begin{equation*}
\sum_{j}\theta _{j} \sum_{k\in\I_j}A_{j,k}\log p_{j,k} \leq \underline{%
\alpha }\sum_{j}\theta _{j} \sum_{k\in\I_j}A_{j,k}\log r_{j,k}.
\end{equation*}
Putting $f_{j,k}=\log p_{j,k}-\underline{\alpha }\log r_{j,k}$, we can
rewrite this required inequality as 
\begin{equation}
\sum_{j=1}^m\theta _{j} \sum_{k\in\I_j} A_{j,k}f_{j,k} \leq 0.
\label{NeededBd}
\end{equation}

By the definition of $\underline{\alpha }$, for all choices of $\sigma
_{j}\in \I_{j}$ we have 
\begin{equation*}
\log p_{1,\sigma _{1}}^{\theta _{1}}\dots p_{m,\sigma _{m}}^{\theta
_{m}}\leq \underline{\alpha }\log r_{1,\sigma _{1}}^{\theta _{1}}\dots
r_{m,\sigma _{m}}^{\theta _{m}},
\end{equation*}%
or, equivalently, $\sum_{j=1}^{m}\theta _{j}f_{j,\sigma _{j}}\leq 0$. As $%
\{A_{j,k}\}_{k\in \I_{j}}$ is a convex combination for each $j$, inequality (%
\ref{NeededBd}) follows directly from this fact. This proves $\dimllc\mu
_{\omega }(x)\geq \underline{\alpha }$ for all $\omega \in \Omega _{1}$, as
claimed.

Similar reasoning establishes the upper bound for $\dimllc\mu _{\omega }(x)$
and that completes the proof that the attainable lower local dimensions lie
in the interval $[\underline{\alpha }$, $\overline{\alpha }]$ for a.a.\ $%
\omega $. The arguments for the upper local dimension are the same and the
only difference with the proof for the local dimension is that we restrict
our attention to $x\in K_{\omega }$ where the local dimension of $\mu
_{\omega }$ at $x$ exists.

Next, we recall that in Proposition \ref{AuxMeas} we saw that for a.a.\ $%
\omega $ there is a probability measure concentrated on $E_{\omega ,\alpha }$
whenever $\alpha =-\beta ^{\prime }(q)$ for some real number $q$. In
particular, $E_{\omega ,\alpha }$ is non-empty for all such $\alpha $ and
thus the set of local dimensions of $\mu _{\omega }$ contains the range of $%
-\beta ^{\prime }(q)$. As $\beta ^{\prime }$ is continuous, Lemma \ref%
{RangeBeta} establishes that the range of $-\beta ^{\prime }$ contains the
interval $(\underline{\alpha }$, $\overline{\alpha })$.

The formula for $\dim _{H}E_{\omega ,\alpha }$, $\dim _{H}\underline{%
E_{\omega ,\alpha }}$ and $\dim _{H}\overline{E_{\omega ,\alpha }}$ for such 
$\alpha $ and a.a.\ $\omega $ were already given in Corollary \ref{DimEalpha}%
.

The endpoints, $\underline{\alpha }$, $\overline{\alpha }$, can also be
easily seen to be attainable local dimensions of $\mu _{\omega }$. To obtain 
$\underline{\alpha }$, for example, choose a minimizing vector $(\sigma
_{1},\dots ,\sigma _{m})$. Each time $\omega _{i}=j$, apply the contraction $%
S_{j,\sigma _{j}}$. Standard probability arguments show that for a.a.\ $%
\omega $ we have $\dimlc\mu _{\omega }(x)=\underline{\alpha }$ for $%
x=\bigcap_{n}S_{\omega _{1},\sigma _{1}}\circ \dots \circ S_{\omega
_{n},\sigma _{n}}(K)\in K_{\omega }$.

Lastly, we check that $\beta $ is convex by proving that $\beta ^{\prime
\prime }\geq 0$. Implicitly differentiating the identity $%
0=\sum_{j=1}^{m}\theta _{j}\log \sum_{k\in \I_{j}}p_{j,k}^{q}r_{j,k}^{\beta
} $ (with $\beta =\beta (q)$) twice and putting $z_{j,k}(q)=\log
p_{j,k}+\beta ^{\prime }(q)\log r_{j,k}$ we obtain

\begin{align*}
& \sum_{j}\theta _{j}\beta ^{\prime \prime }\frac{\left( \sum_{k\in \I%
_{j}}p_{j,k}^{q}r_{j,k}^{\beta }\log r_{j,k}\right) \left( \sum_{k\in \I%
_{j}}p_{j,k}^{q}r_{j,k}^{\beta }\right) }{\left( D_{j}(q)\right) ^{2}} \\
& =\sum_{j}\theta _{j}\left[ \frac{\left( \sum_{k\in \I%
_{j}}p_{j,k}^{q}r_{j,k}^{\beta }z_{j,k}(q)\right) ^{2}-\sum_{k\in \I%
_{j}}p_{j,k}^{q}r_{j,k}^{\beta }\left( z_{j,k}(q)\right) ^{2}\left(
\sum_{k\in \I_{j}}p_{j,k}^{q}r_{j,k}^{\beta }\right) }{\left(
D_{j}(q)\right) ^{2}}\right] .
\end{align*}%
As the coefficients of $\beta ^{\prime \prime }$ in the formula above are
all negative, it will be enough to show 
\begin{equation}
\left( \sum_{k\in \I_{j}}p_{j,k}^{q}r_{j,k}^{\beta }z_{j,k}(q)\right)
^{2}-\sum_{k\in \I_{j}}p_{j,k}^{q}r_{j,k}^{\beta }\left( z_{j,k}(q)\right)
^{2}\left( \sum_{k\in \I_{j}}p_{j,k}^{q}r_{j,k}^{\beta }\right) \leq 0\text{ 
}  \label{beta''Formula}
\end{equation}%
for each $j$ and $q$. Expanding (\ref{beta''Formula}) yields 
\begin{eqnarray*}
&&\sum_{i,k\in \I_{j}}p_{j,k}^{q}r_{j,k}^{\beta
}z_{j,k}p_{j,i}^{q}r_{j,i}^{\beta }z_{j,i}-\sum_{i,k\in \I%
_{j}}p_{j,k}^{q}r_{j,k}^{\beta }p_{j,i}^{q}r_{j,i}^{\beta }z_{j,k}^{2} \\
&=&\sum_{i<k}p_{j,k}^{q}r_{j,k}^{\beta }p_{j,i}^{q}r_{j,i}^{\beta
}(2z_{j,k}z_{j,i}-(z_{j,k}^{2}+z_{j,i}^{2})).
\end{eqnarray*}%
Since $2z_{j,k}(q)z_{j,i}(q)\leq z_{j,k}^{2}(q)+z_{j,i}^{2}(q)$ for each $q$%
, this completes the proof that $\beta ^{\prime \prime }\geq 0$, hence $%
\beta $ is convex.
\end{proof}

\begin{corollary}
For almost all $\omega $ and $\mu _{\omega }$ almost all $x$, $\dimllc\mu
_{\omega }(x)=\dim _{H}K_{\omega }$.
\end{corollary}

\begin{proof}
The strict convexity of $\beta $ proves there is a unique maximum value of
the function $f_{\omega }(\alpha )=\dim _{H}E_{\omega ,\alpha
}=\inf_{q}(q\alpha +\beta (q))$. By differentiating one can check this
occurs at $q=0$, $\alpha =-\beta ^{\prime }(0)$. By the mass distribution
principle, $\beta (0)=f_{\omega }(\alpha )=\dim _{H}E_{\omega ,\alpha
}=\alpha $ when $\alpha =-\beta ^{\prime }(0)$. But according to formulas (%
\ref{Hdim}) and (\ref{DefBeta2}), $\beta (0)=\dim _{H}K_{\omega }$ for a.a.\ 
$\omega$.
\end{proof}

\section{The Random Finite Type Condition}

\label{sec:equicontractivefinitetype}

\subsection{Finite type}

In the previous section we concerned ourselves with the multifractal
spectrum under the assumption of the uniform strong separation condition.
For the remainder of the paper, we will relax that assumption and
investigate the random analogue of the finite type condition.

For this, we will restrict our attention to random homogeneous iterated
function systems $(\fS,\Prob)$ acting on $\BbR$ that are equicontractive,
that is, there is some $0<r<1$ such that each contraction is of the form 
\begin{equation*}
S_{j,k}(x)=rx+d_{j,k}.
\end{equation*}
We will also assume that for all $\omega \in \Omega$ that the convex hull of
the random attractor $K_\omega$ is $[0,1]$. Equivalently for each $j$ the
convex hull of $K_j$ associated to $\fS_j$ is $[0,1]$. In particular, $%
S_{j,k}([0,1])\subseteq \lbrack 0,1]$ for all $j,k$ and for each $j$ there
is some $k,\ell \in \I_{j}$ such that $S_{j,k}(0)=0$ and $S_{j,\ell }(1)=1$.

To ease notation, given $\omega \in \Omega $ and a finite coding $\sigma
=(\sigma _{1},\sigma _{2},\dots ,\sigma _{n})\in \Lambda _{\omega ,n}$ we
concisely write $S_{\omega ,\sigma }$ for the composition 
\begin{equation*}
S_{\omega ,\sigma }=S_{\omega _{1},\sigma _{1}}\circ S_{\omega _{2},\sigma
_{2}}\circ \dots S_{\omega _{n},\sigma _{n}}.
\end{equation*}
Note that the equicontractive and convex hull assumptions ensure that $%
S_{\omega ,\sigma }(K_{\omega })$ has diameter $r^{n}$ whenever $\sigma $ is
of length $n$.

The notion of finite type was originally introduced by Ngai and Wang in \cite%
{NW} for a single IFS. Here we extend the definition to random IFS.

\begin{definition}
\label{defn:finite type} Let $\fS=\{\fS_{1},\dots ,\fS_{m}\}$ be an
equicontractive RIFS with contraction ratio $r$. We say that the RIFS is of 
\textbf{finite type} if there exists a finite set $F$ such that for all
choices of $n\in \mathbb{N}$ , $\omega \in \Omega $ and $\sigma ,\tau \in
\Lambda _{\omega ,n}$ we have 
\begin{equation*}
\left\vert S_{\omega ,\sigma }(0)-S_{\omega ,\tau }(0)\right\vert r^{-n}>%
\diam K_{\omega }=1
\end{equation*}
or 
\begin{equation*}
(S_{\omega ,\sigma }(0)-S_{\omega ,\tau }(0))r^{-n}\in F.
\end{equation*}
\end{definition}

Clearly, a RIFS that satisfies the uniform strong separation condition is of
finite type. However, the converse is not true as Example \ref{ExFT} shows.

\subsection{Net intervals, characteristic vectors and symbolic
representations}

As explained in \cite{F3,F1,F2,HHM}, an iterated function system of finite
type generates a geometric structure that has useful properties for studying
the local dimension theory of the associated self-similar measures. Our
first step is to extend these structures to the random setting.

\begin{definition}
Let $\fS$ be a random iterated function system and let $\omega \in \Omega $.
For each positive integer $n$, let $h_{1},\dots ,h_{s_{n}}$ be the elements
of the set 
\begin{equation*}
\{S_{\omega ,\sigma }(z)\;:\;z\in \{0,1\}\text{ and }\sigma \in \Lambda
_{\omega ,n}\},
\end{equation*}
listed in increasing order. Put 
\begin{equation*}
\F_{\omega ,n}=\{[h_{i},h_{i+1}]:1\leq i\leq s_{n}-1\text{ and }
(h_{i},h_{i+1})\cap K_{\omega }\neq \emptyset \}\text{.}
\end{equation*}
Elements of $\F_{\omega ,n}$ will be called $\omega $-\textbf{net intervals
of level $n$}. The interval $[0,1]$ is understood to be the only net
interval of level $0$.
\end{definition}

Let $\omega \in \Omega $. For each $\Delta \in \F_{\omega ,n}$, where $n\geq
1$, there exists a unique element $\widehat{\Delta }\in \F_{\omega ,n-1}$
which contains $\Delta $, called the \textit{parent} (of \textit{child} $%
\Delta )$. Given $\Delta =[a,b]\in \F_{\omega ,n}$, we denote the \textit{%
normalized length} of $\Delta $ by 
\begin{equation*}
\ell _{n}(\Delta )=r^{-n}(b-a).
\end{equation*}
By the $\omega $-\textit{neighbour set} of $\Delta $ we mean the ordered
tuple 
\begin{equation*}
V^{\omega}_{n}(\Delta )=(a_{1},a_{2},\dots ,a_{J}),
\end{equation*}
where $a_{i}<a_{i+1}$ and there is some $\sigma \in \Lambda _{\omega ,n}$
such that $r^{-n}(a-S_{\omega ,\sigma }(0))=a_{i}$ for each $i$.
Equivalently, there exists $\sigma \in \Lambda _{\omega ,n}$ such that $%
S_{\omega ,\sigma }(x)=r^{n}(x-a_{i})+a$.

Suppose $\Delta \in \F_{\omega ,n}$ has parent $\widehat{\Delta }$. It is
possible for $\widehat{\Delta }$ to have multiple children with the same
normalized length and the same $\omega $-neighbour set as $\Delta $. Order
these equivalent children from left to right as $\Delta _{1},\Delta
_{2},\dots ,\Delta _{k}$. We denote by $t_{n}(\Delta )$ the integer $t$ such
that $\Delta _{t}=\Delta $.

\begin{definition}
The $\omega $-\textbf{characteristic vector} \textit{of} $\Delta \in \F%
_{\omega ,n}$ is defined to be the triple 
\begin{equation*}
\C_{n}^{\omega }(\Delta )=(\ell _{n}(\Delta ),V^{\omega}_{n}(\Delta
),t_{n}(\Delta )).
\end{equation*}
\end{definition}

Similar to the case of a single IFS, the normalized length and neighbour set
of any child of the $\omega $-net interval $\Delta $ of level $n$ depend
only on the normalized length and neighbour set of $\Delta $ and the
contractions of the IFS $\fS\omega _{n+1}$.

By the $\omega $-\textit{symbolic representation} of a net interval $\Delta
\in \F_{\omega ,n}$ we mean the $n+1$ tuple $(\C_{0}^{\omega }(\Delta
_{0}),\dots ,\C_{n}^{\omega }(\Delta _{n}))$ where $\Delta _{0}=[0,1]$, $%
\Delta _{n}=\Delta $, and for each $j=1,\dots ,n$, $\Delta _{j-1}$ is the
parent of $\Delta _{j}$. Similarly, for each $x\in K_{\omega }$ the $\omega $%
-\textit{symbolic representation }of $x$ will be the sequence of
characteristic vectors 
\begin{equation*}
\lbrack x]_{\omega }=(\C_{0}^{\omega }(\Delta _{0}),\C_{1}^{\omega }(\Delta
_{1}),\dots )
\end{equation*}
where $x\in \Delta _{n}\in \F_{\omega ,n}$ for each $n$ and $\Delta
_{j-1}\in \F_{\omega ,n-1}$ is the parent of $\Delta _{j}$. The symbolic
representation uniquely determines $x$ and is unique unless $x$ is the
endpoint of some net interval, in which case there can be two different
symbolic representations. We will write $\Delta_{\omega,n}(x)$ for an $%
\omega $-net interval of level $n$ containing $x$.

\subsection{A characterization of finite type}

As with a single IFS, finite type is characterized by the property that
there are only finitely many characteristic vectors.

\begin{theorem}
An equicontractive RIFS is of finite type if and only if there are only
finitely many characteristic vectors (taken over all choices of $\omega $).
\end{theorem}

\begin{proof}
First, assume that the RIFS is of finite type. Let $\Delta =[a,b]$ be an $%
\omega $-net interval of level $n$. From the definition, we see that there
exist contractions $S_{\omega ,\sigma }$ and $S_{\omega ,\tau }$ with $%
\sigma ,\tau \in \Lambda _{\omega ,n}$, and $c,d\in \{0,1\}$ such that $%
a=S_{\omega ,\sigma }(c)$ and $b=S_{\omega ,\tau }(d)$. If $c=d$, then $%
r^{-n}(b-a)\in F$, of which there are only finitely many choices. If $c=0$
and $d=1$, then we see that $r^{-n}(S_{\omega ,\sigma }(1)-S_{\omega ,\tau
}(1))\in F$, again of which there are only finitely many choices. Lastly, if 
$c=1$ and $d=0$ there exists some $\eta \in \Lambda _{\omega ,n}$ such that $%
S_{\omega ,\eta }(0)<S_{\omega ,\sigma }(1)<S_{\omega ,\tau }(0)<S_{\omega
,\eta }(1)$, for otherwise $(a,b)\cap K_{\omega }=\emptyset $. There are
only finitely many choices for $r^{-n}(S_{\omega ,\tau }(0)-S_{\omega ,\eta
}(0))$ and $r^{-n}(S_{\omega ,\eta }(1)-S_{\omega ,\sigma }(1))$, and
further, $r^{-n}(S_{\omega ,\eta }(1)-S_{\omega ,\eta }(0))$ is fixed. This
shows that there are only a finite number of choices for $\ell _{n}(\Delta )$%
.

By a similar logic to the above, we see that there are only a finite number
of $\eta \in \Lambda _{\omega ,n}$ such that $S_{\omega ,\eta }(0)\leq
a<b\leq S_{\omega ,\eta }(1)$ and this proves there are only a finite number
of possible neighbour sets. It follows that there are only a finite number
of characteristic vectors.

The other direction is similar.
\end{proof}

Since one choice of $\omega $ is the constant sequence, it is immediate that
if $\fS=\{\fS_{1},\fS_{2},\dots ,\fS_{m}\}$ is of finite type, then so is
each individual $\fS_{j}$. We do not know whether the converse holds.

We provide a sufficient condition for finite type in the proposition below.

\begin{proposition}
\label{prop:4.6} Let $\fS=\{\fS_{1},\fS_{2},\dots ,\fS_{m}\}$ be a RIFS
where $\fS_{j}=\{S_{j,k}:k\in \I_{j}\}$. If the IFS consisting of all the
contractions $\{S_{j,k}:k\in \I_{j},j=1,\dots,m\}$ is of finite type, then
so is $\fS$.
\end{proposition}

\begin{proof}
For any $\omega \in \Omega $ and $\sigma $, $\tau \in \Lambda _{\omega ,n}$,
the points $S_{\omega ,\sigma }(0)$ and $S_{\omega ,\tau }(0)$ are images of 
$0$ under contractions of the IFS $\{S_{j,k}:k\in \I_{j},j=1,\dots,m\}$. As
this IFS is of finite type, these images are either sufficiently far apart,
or their normalized difference belongs to a finite set. Hence there are only
a finite number of normalized `large' differences for $S_{\omega ,\sigma
}(0) $ and $S_{\omega ,\tau }(0)$. This proves the result.
\end{proof}

\begin{example}
\label{ExFT} Let $\rho$ be a Pisot number, that is a real algebraic integer
greater than 1, all of whose conjugates are strictly less than 1 in modulus.
Let $\fS_j$ consist of contractions $S_{j,k} = x/\rho + \beta_{j,k}$ where $%
\beta_{j,k} \in \mathbb{Q}(\rho)$, with $k = 0, \dots, \mathcal{N}(j) $.
Then by \cite[Theorem 2.9]{NW} we have that the union of all $S_{j,k}$
satisfies the finite type condition. Hence the RIFS $\fS=\{\fS_{1},\fS%
_{2},\dots ,\fS_{m}\}$ is of finite type. In this case the individual $\fS%
_{j}$ do not need to satisfy the strong separation condition.
\end{example}

\subsection{Local Dimensions of measures of finite type}

\label{sec:local}

Recall that we assume that the convex hull of each $K_{\omega }$ is $[0,1]$.

Provided the local dimension exists at some $x\in K_{\omega }=\supp\mu
_{\omega }$, it is easy to see we have the formula 
\begin{equation}
\dimlc\mu _{\omega }(x)=\lim_{n\rightarrow \infty }\frac{\log \mu _{\omega
}([x-r^{n},x+r^{n}])}{n\log r}\text{ }  \label{eq:locdim}
\end{equation}
and similarly for upper and lower local dimensions.

In the case of a single IFS satisfying the finite type condition, it was
shown in \cite{HHM, HHN} that the local dimension can be computed using net
intervals and what are known as transition matrices. Similar statements can
be made in the random setting, as we now explain. We first note some
additional notation and facts about $\mu _{\omega }$-measures of net
intervals.

Given $\omega \in \Omega $ and $\sigma =(\alpha _{1},\dots,\sigma _{n})\in
\Lambda _{\omega ,n}$ we put 
\begin{equation*}
p_{\omega ,\sigma }=\prod_{i=1}^{n}p_{\omega _{i},\sigma _{i}}.
\end{equation*}

We write $f_{n}\sim g_{n}$ when we mean there are positive constants $%
c_{1},c_{2}$ such that 
\begin{equation*}
f_{n}\leq c_{1}g_{n}\leq c_{2}f_{n},
\end{equation*}
for all $n$.

\begin{lemma}
Let $\omega \in \Omega $ and suppose $\Delta =[a,b]\in \F_{\omega ,n}$ has
normalized length $\ell _{n}(\Delta )$ and $\omega $-neighbour set $%
V^{\omega}_{n}(\Delta )=(a_{1},\dots,a_{J})$. Then, 
\begin{equation}
\mu _{\omega }(\Delta )=\sum_{i=1}^{J}\mu ([a_{i},a_{i}+l_{n}(\Delta ]) 
\hspace{-0.5cm}\sum _{\substack{ \sigma \in \Lambda _{\omega ,n}  \\ %
r^{-n}(a-S_{\omega ,\sigma }(0))=a_{i}}} \hspace{-0.6cm}p_{\omega ,\sigma
}\quad\sim\;\;P_{\omega ,n}(\Delta ),  \label{measDelta}
\end{equation}
where $P_{\omega ,n}(\Delta )=\sum_{i=1}^{J}P_{\omega ,n}^{i}(\Delta )$ and
for $i=1,\dots,J$, 
\begin{equation*}
P_{\omega ,n}^{i}(\Delta )=\hspace{-0.5cm}\sum_{\substack{ \sigma \in
\Lambda _{\omega ,n}  \\ r^{-n}(a-S_{\omega ,\sigma }(0))=a_{i}}}\hspace{%
-0.5cm}p_{\omega ,\sigma }\text{ }.
\end{equation*}
\end{lemma}

\begin{proof}
This follows as in \cite[Section 3]{HHM}.
\end{proof}

We can calculate $P_{\omega ,n}(\Delta )$ by means of transition matrices
which take us from level $n-1$ to level $n$.

\begin{definition}
Let $\Delta =[a,b]\in \F_{\omega ,n}$ and let $\widehat{\Delta }=[c,d]\in \F%
_{\omega ,n-1}$ denote its parent net interval. Assume $V_{n}^{\omega
}(\Delta )=(a_{1},\dots,a_{J})$ and $V_{n-1}^{\omega }(\widehat{\Delta }%
)=(c_{1},\dots,c_{I})$. The \textbf{primitive transition matrix, }denoted by 
$T_{\omega }(\C_{n-1}^{\omega }(\widehat{\Delta }),\C_{n}^{\omega }(\Delta
)), $ is the $I\times J$ matrix whose $(i,j)$'th entry, $T_{ij}$, is defined
as follows: Put $T_{ij}=p_{\ell }$ if there is a letter $\ell $ and coding $%
\sigma \in \Lambda _{\omega ,n-1}$ with $\sigma \ell \in \Lambda _{\omega
,n} $ satisfying $S_{\sigma \ell }(0)=a-r^{n}a_{j}$ and $S_{\sigma
}(0)=c-r^{n-1}c_{i}$. If there is no such $\ell $, we put $T_{ij}=0$.
\end{definition}

Every column of a primitive transition matrix has at least one non-zero
entry by virtue of the existence of $c_{i}$. When the support of $\mu
_{\omega }$ is the full interval $[0,1]$, then it also follows that each row
has a non-zero entry.

The same reasoning as in \cite{HHM} shows that if $Q_{\omega ,n}(\Delta
)=(P_{\omega ,n}^{1}(\Delta ),\dots ,P_{\omega ,n}^{J}(\Delta ))$ and $%
Q_{0}([0,1])=(1)$, then 
\begin{equation*}
Q_{\omega ,n}(\Delta )=Q_{\omega ,n-1}(\widehat{\Delta })T_{\omega }(\C%
_{n-1}^{\omega }(\widehat{\Delta }),\C_{n}^{\omega }(\Delta )).
\end{equation*}
Consequently, if $\Delta $ has $\omega $-symbolic representation $(\gamma
_{0},\gamma _{1},\dots ,\gamma _{n})$, then 
\begin{equation*}
P_{\omega ,n}(\Delta )=\left\Vert T_{\omega }(\gamma _{0},\gamma _{1})\dots
T_\omega(\gamma _{n-1},\gamma _{n})\right\Vert
\end{equation*}
where the matrix norm is given by $\left\Vert (M_{ij})\right\Vert
=\sum_{ij}\left\vert M_{ij}\right\vert $. We will write $T_{\omega }(\gamma
_{0},\gamma _{1},\dots ,\gamma _{n})$ for $T_{\omega }(\gamma _{0},\gamma
_{1})\dots T_{\omega }(\gamma _{n-1},\gamma _{n})$ and call this a \textit{%
transition matrix}.

If $A$ and $B$ are transition matrices, then since $A$ has a non-zero entry
in each column it is easy to see that $\left\Vert AB\right\Vert \geq
c\left\Vert B\right\Vert $, where $c>0$ depends only on $A$. Thus if $%
[x]_{\omega }=(\gamma _{0},\gamma _{1},\dots )$ and $\Delta _{\omega,n}$ is
the $\omega $-net interval containing $x$ with symbolic representation $%
(\gamma _{0},\gamma _{1},\dots ,\gamma _{n})$, then for any $N\leq n$. 
\begin{equation*}
\mu _{\omega }(\Delta_{\omega,n})\sim P_{\omega ,n}(\Delta_{\omega,n})\sim
\left\Vert T_{\omega }(\gamma _{N},\gamma _{N+1},\dots ,\gamma
_{n})\right\Vert ,
\end{equation*}
where the constants of comparability can be chosen to depend only on $N$.

Given $x\in K_{\omega }$ and the $\omega $-net interval $\Delta
_{\omega,n}(x)$, we let $\Delta _{\omega,n}^{+}(x)$ and $\Delta_{%
\omega,n}^{-}(x)$ be the right and left adjacent $\omega $-net intervals of
level $n$ (understanding that some of these might be empty if $x=0,1$ or if $%
K_{\omega }$ is not an interval). Put 
\begin{equation*}
M_{\omega ,n}(x)=\mu _{\omega }(\Delta_{\omega,n}(x))+\mu _{\omega }(\Delta
_{\omega,n}^{+}(x))+\mu _{\omega }(\Delta_{\omega,n}^{-}(x)).
\end{equation*}

\begin{proposition}
Let $\mu _{\omega }$ be a random self-similar measure associated with an
equicontractive RIFS of finite type. For each $x$ in the random attractor $%
K_{\omega }$ we have the formula 
\begin{equation*}
\dimlc\mu _{\omega }(x)=\lim_{n\rightarrow \infty }\frac{\log M_{\omega
,n}(x)}{n\log r}
\end{equation*}
provided the limit exists. Similar formulas hold for the upper and lower
local dimensions.
\end{proposition}

\begin{proof}
The finite type property guarantees that all net intervals of level $n$ have
length comparable to $r^{n}$. This fact is enough to ensure that $\mu
_{\omega }([x-r^{n},x+r^{n}])$ is comparable to $M_{\omega ,n}(x)$. For more
details, see \cite[Theorem 2.6]{HHN}.
\end{proof}

A self-similar measure associated with an IFS $\{S_{j}(x)=rx+d_{j}%
\}_{j=1}^{N}$ with $d_{j}<d_{j+1}$ and probabilities $\{p_{j}:j=1,\dots,N\}$
is said to be regular if $p_{0}=p_{N}=\min p_{j}$. We make a similar
definition in the RIFS case.

\begin{definition}
We will say that the random self-similar measure associated with an
equicontractive RIFS and probabilities $\{p_{j,k}:k\in \mathcal{I}_{j}\}$, $%
j=1,\dots,m$, is \textbf{regular} if $p_{j,0}=p_{j,\mathcal{N}%
(j)}=\min_{k}p_{j,k}$ for each $j$.
\end{definition}

In the single IFS case, regular self-similar measures $\mu $ have the
property that the $\mu $-measure of adjacent (non-empty) net intervals are
comparable (see \cite[Lemma 3.5]{HHM}). The same proof gives the same
conclusion in the RIFS case and therefore $M_{\omega ,n}(x)\sim \mu _{\omega
}(\Delta _{\omega,n}(x))$ for all $x$. This gives the following simple
formula for local dimensions of regular self-similar measures.

\begin{theorem}
\label{regLocDim}Let $\mu _{\omega }$ be a random self-similar measure
associated with an equicontractive RIFS of finite type. Assume that $\mu
_{\omega }$ is regular and that $x\in $ $K_{\omega }$ has $\omega $-symbolic
representation $[x]_{\omega }=(\gamma _{0},\gamma _{1},\dots)$. For any $%
N\in \mathbb{N}$, 
\begin{eqnarray*}
\dimlc\mu _{\omega }(x) &=&\lim_{n\rightarrow \infty }\frac{\log \mu_\omega
(\Delta _{\omega,n}(x))}{n\log r} \\
&=&\lim_{n\rightarrow \infty }\frac{\log \left\Vert T_\omega(\gamma
_{N},\gamma _{N+1},\dots ,\gamma _{N+n})\right\Vert }{n\log r}
\end{eqnarray*}
provided the limit exists. Similar formulas hold for the upper and lower
local dimensions.
\end{theorem}

We can obtain bounds on the local dimensions by making estimates on the
norms of transition matrices. For this, we introduce further notation: Given
a matrix $T=(T_{ij})$, denote by $\left\Vert T\right\Vert _{c,\min }$ and $%
\left\Vert T\right\Vert _{c,\max }$ the pseudo-norms 
\begin{equation*}
\left\Vert T\right\Vert _{c,\min }=\min_{j}\sum_{i}\left\vert
T_{ij}\right\vert ,\text{ }\left\Vert T\right\Vert _{c,\max
}=\max_{j}\sum_{i}\left\vert T_{ij}\right\vert ,
\end{equation*}
that is, the minimum and maximum column sums. For matrices with non-negative
entries we have 
\begin{equation*}
\left\Vert T_{1}T_{2}\right\Vert _{c,\min }\geq \left\Vert T_{1}\right\Vert
_{c,\min }\left\Vert T_{2}\right\Vert _{c,\min },\left\Vert
T_{1}T_{2}\right\Vert _{c,\max }\leq \left\Vert T_{1}\right\Vert _{c,\max
}\left\Vert T_{2}\right\Vert _{c,\max }\text{ }
\end{equation*}
\begin{equation*}
\text{and }\left\Vert T\right\Vert _{c,\min }\leq \left\Vert T\right\Vert
_{c,\max }\leq \left\Vert T\right\Vert \leq C\left\Vert T\right\Vert
_{c,\max },
\end{equation*}
where $C$ is the number of columns of $T$. We can similarly define the
minimum and maximum row sum pseudo-norms, $\left\Vert T\right\Vert _{\text{%
row},\min }$ and $\left\Vert T\right\Vert _{\text{row},\max }$.

\begin{corollary}
Let $x\in K_{\omega }$ with $[x]_{\omega }=(\gamma _{0},\gamma _{1},\dots)$.
For each integer $N$, let $m_{N}=\min \{\left\Vert T(\gamma _{j},\gamma
_{j+1})\right\Vert _{c,\min }:j\geq N\}$ and $M_{N}=\max \{\left\Vert
T(\gamma _{j},\gamma _{j+1})\right\Vert _{c,\max }:j\geq N\}$. Then 
\begin{equation*}
\sup_{N}\frac{\log M_{N}}{\log r}\leq \dimlc\mu _{\omega }(x)\leq \inf_{N} 
\frac{\log m_{N}}{\log r}.
\end{equation*}
\end{corollary}

We remark that $m_{N}$ and $M_{N}$ exist as there are only finitely many
characteristic vectors.

\subsection{The Essential Class}

\begin{definition}
Let $\fS$ be an equicontractive RIFS of finite type and suppose $\gamma $
and $\beta $ are two characteristic vectors. Let $\nu =(\nu _{1},\dots,\nu
_{k})\in \mathcal{A}^{k}$. We say that $\gamma =\gamma _{0},\gamma
_{1},\dots,\gamma _{k}=\beta $ is an \textbf{admissible }$\nu $\textbf{-path
linking }$\gamma $ \textbf{to} $\beta $ if there is some $\omega \in \Omega $
with substring $\nu $, say $\omega _{N+j}=\nu _{j}$ for $j=1,\dots,k$, and $%
\omega $-net intervals $\Delta _{j}$ of level $N+j$ for $j=0,\dots,k$, where 
$\Delta _{j}$ has characteristic vector $\gamma _{j}$ and each $\Delta _{j}$
is a child of $\Delta _{j-1}$. In this case, we say that $\beta $ is a 
\textbf{$\nu$-descendant} of $\gamma $.

If there is some finite word $\nu $ and admissible $\nu $-path linking $%
\gamma $ to $\beta $ we say there is an \textbf{admissible path linking }$%
\gamma $ \textbf{to} $\beta $ and that $\beta $ is a \textbf{descendant} of $%
\gamma $. We call $\beta $ a \textbf{child} of $\gamma $ if there is an
admissible path of length one.
\end{definition}

We show, first, that descendancy is a transitive relationship.

\begin{lemma}
If $\gamma _{2}$ is a descendant of $\gamma _{1}$, and $\gamma _{1}$ is a
descendant of $\gamma _{0}$, then $\gamma _{2}$ is a descendant of $%
\gamma_{0}$.
\end{lemma}

\begin{proof}
To simplify notation we will assume both descendants are children; the
arguments are the same in general.

Assume $\Delta _{0}$ is an $\omega $-net interval of level $n-1$ with
characteristic vector $\gamma _{0}$, $\Delta _{1}$ is an $\omega $-net
subinterval of level $n$ with characteristic vector $\gamma _{1}$, $\Delta
_{1}^{\prime }$ is a $\nu $-net interval of level $k$ with characteristic
vector $\gamma _{1}$ and $\Delta _{2}^{\prime }$ is a $\nu $-net interval of
level $k+1$ with characteristic vector $\gamma _{2}$. Let $\tau =(\omega
_{1},\dots,\omega _{n},\nu _{k+1},\nu _{k+2},\dots)\in \Omega $. The
intervals $\Delta _{0}$ and $\Delta _{1}$ are clearly also $\tau $-net
intervals of level $n-1$ and $n$, with characteristic vectors $\gamma _{0}$
and $\gamma _{1}$ respectively. Furthermore, since the normalized length and
neighbour set of any $\tau $-child of $\Delta _{1}$ depends only on the
normalized length and neighbour set of $\Delta _{1}$ (which coincides with
those of $\Delta _{1}^{\prime }$) and the contractions in $\fS_{\tau _{n+1}}$
which are the same as $\fS_{\nu _{k+1}}$, these coincide with the $\nu $%
-children of $\Delta _{1}^{\prime }$.

In particular, there is a subinterval $\Delta _{2}$ of $\Delta _{1}$, which
will be a $\tau $-net interval of level \thinspace $n+1$ with characteristic
vector $\gamma _{2}$ provided $\inte\Delta _{2}\cap K_{\tau }\neq \emptyset $%
. We know that $\inte\Delta _{2}^{\prime }\cap K_{\nu }\neq \emptyset $ and
that means there is some $\sigma ^{\prime }\in \Lambda _{\nu ,k}$, $N\in 
\mathbb{N}$ and $\sigma _{N}=(j_{1},\dots,j_{N})$ with $j_{i}\in \mathcal{I}%
_{\nu _{k+i}}$, $i=1,\dots,N$, such that $S_{\nu ,\sigma ^{\prime }}([0,1])$
covers $\Delta _{1}^{\prime }$ and $S_{\nu ,\sigma ^{\prime }\sigma
_{N}}([0,1])\subseteq \inte\Delta _{2}^{\prime }$. But because the geometry
of the pair $\Delta _{1}$, $\Delta _{2}$ is identical to that of the pair $%
\Delta _{1}^{\prime }$, $\Delta _{2}^{\prime }$ up to rescaling, it follows
that there is some $\sigma \in \Lambda _{\omega ,n}$ such that $S_{\omega
,\sigma }([0,1])=S_{\tau ,\sigma }([0,1])$ covers $\Delta _{1}$ and $S_{\tau
,\sigma \sigma _{N}}([0,1])\subseteq \inte\Delta _{2}$. This ensures $\inte%
\Delta _{2}\cap K_{\tau }\neq \emptyset $.

Thus $\gamma _{2}$ is the characteristic vector of the $\tau $-net interval $%
\Delta _{2}$ and that proves $\gamma _{2}$ is a descendant of $\gamma _{0}$.
\end{proof}

\begin{definition}
A non-empty subset $\mathcal{L}$ of the set of all characteristic vectors is
called a \textbf{loop class} if whenever $\gamma ,\beta \in \mathcal{L}$
then there is an admissible path linking $\gamma $ to $\beta $, say $\gamma
=\gamma _{0},\gamma _{1},\dots,\gamma _{k}=\beta $, with each $\gamma
_{j}\in \mathcal{L}$. If every child of every member of a loop class $%
\mathcal{L}$ is again in $\mathcal{L}$, we call $\mathcal{L}$ an \textbf{%
essential class}.
\end{definition}

Feng, in \cite{F3}, proved that each IFS of finite type admits a unique
essential class. We will prove that this is true in the RIFS case, as well.
First, we present a result that may be of independent interest.

\begin{lemma}
There is a characteristic vector that is a descendant of every
characteristic vector.
\end{lemma}

\begin{proof}
Choose a characteristic vector $\beta $ with minimal normalized length and
amongst the characteristic vectors of minimal length, choose one with the
largest number of neighbours. We claim that $\beta $ is a descendant of all
characteristic vectors.

Assume $\beta $ is the $\omega $-characteristic vector of $\Delta _{0}\in 
\mathcal{F}_{\omega ,n}$. Suppose $\Delta _{0}$ has endpoints $S_{\omega
,\sigma _{1}}(z_{1})$, $S_{\omega ,\sigma _{2}}(z_{2})$ where $%
z_{1},z_{2}\in \{0,1\}$ and that $z_{0}\in K_{\omega }$ is in the interior
of $\Delta $. Let $\gamma $ be any characteristic vector and assume $\gamma $
is the $\nu $-characteristic vector of net interval $\Delta _{1}\in \mathcal{%
F}_{\nu ,m}$. As the interior of $\Delta _{1}$ has non-empty intersection
with $K_{\nu }$, we can choose a suitable coding $\sigma \in \Lambda _{\nu
,N}$ of length $N$ such that $z_{0}\in S_{\nu ,\sigma }([0,1])\subseteq
\Delta _{1}$.

Now consider $\tau =(\nu _{1},\dots,\nu _{N},\omega )\in \Omega $. If $%
S_{\tau ,\sigma \sigma _{j}}(z_{j})=S_{\nu ,\sigma }\circ S_{\omega ,\sigma
_{j}}(z_{j})=h_{j}$ for $j=1,2$, then $[h_{1},h_{2}]=S_{\nu ,\sigma }(\Delta
_{0})$ is a subset of $\Delta _{1}$ and contains the point $S_{\nu ,\sigma
}(z_{0})$ which belongs to $K_{\tau }$. Because of the minimality assumption
of the length, $[h_{1},h_{2}]$ is a $\tau $-net interval. It has at least as
many neighbours as $\Delta _{0}$, but cannot have more by the maximality
assumption. Consequently, it has the same neighbours as $\Delta _{0}$ and
hence has the same (reduced) characteristic vector $\beta $. Consequently, $%
\beta $ is a descendant of $\gamma $.{}
\end{proof}

\begin{proposition}
There exists an essential class and it is unique.
\end{proposition}

\begin{proof}
Let $\mathcal{E}$ consist of the characteristic vector $\beta $ of minimal
normalized length and maximal number of neighbours, as in the previous
lemma, together with all its descendants.

If $\gamma _{1},\gamma _{2}\in \mathcal{E}$, then both are descendants of $%
\beta $. But $\beta $ is a descendant of both $\gamma _{1}$ and $\gamma _{2}$
and hence by transitivity each is a descendant of the other. Moreover, any
descendant of $\gamma _{1}$ is also a descendant of $\beta $ and hence
belongs to $\mathcal{E}$. This proves $\mathcal{E}$ is an essential class.

The essential class is unique since $\beta $, being a descendant of every
characteristic vector, will belong to any essential class.
\end{proof}

Assume the essential class is
\begin{equation*}
\mathcal{E}=\{\gamma _{j}:j=1,\dots,N\}.
\end{equation*}
For each $k=1,\dots,m$, let $A_{k}$ be the $N\times N$ `adjacency' matrix
whose $(i,j)$'th entry is $1$ if $\gamma _{j}$ is a child of $\gamma _{i}$
under the action of $\fS_{k}$ and $0$ otherwise. Let $e_{j}$ be the $N$%
-vector with $1$ in position $j$ and 0 else.

\begin{proposition}
\label{prop:essential} Let $\omega \in \Omega $ and assume $\gamma _{j}\in%
\mathcal{E}$ is the characteristic vector of an $\omega $-net interval $%
\Delta _{j}$ of level $J_{j}$. The upper box-counting dimension of $%
K_{\omega }$ is equal to 
\begin{equation*}
\limsup_{n\rightarrow \infty }\frac{\log \lVert e_{j}A_{\omega
_{J_{j}+1}}\dots A_{\omega _{J_{j}+n}}\rVert }{n\lvert\log r\rvert}.
\end{equation*}
The lower box-counting dimension is similar.
\end{proposition}

\begin{proof}
Choose $\sigma $ such that $S_{\omega ,\sigma }([0,1])\cap K_{\omega
}\subseteq \Delta _{j}\cap K_{\omega }\subseteq K_{\omega }$. It follows
from these inclusions that the dimensions of $K_{\omega }$ and $\Delta
_{j}\cap K_{\omega }$ coincide.

Now $\Delta _{j}\cap K_{\omega }$ is covered by the net intervals of level $%
J_{j}+n$ that are descendants of $\gamma _{j}$ under the finite string $%
\omega _{J_{j}+1},\dots,\omega _{J_{j}+n}$, and these have length $\sim
r^{J_{j}+n}$. The number of such descendants is equal to the sum of the
entries of the $j$'th row of $A_{\omega _{J_{j}+1}}\dots A_{\omega
_{J_{j}+n}}$, in other words, $\lVert e_{j}A_{\omega _{J_{j}+1}}\dots
A_{\omega _{J_{j}+n}}\rVert $.

This shows the upper and lower box-counting dimensions are as claimed.
\end{proof}

We recall the statement of Kingman's subadditive ergodic theorem. We will
use it to prove the almost sure existence of the limit.

\begin{proposition}[Kingman's subadditive ergodic theorem]
\label{prop:Kingman} Let $F$ be a measure preserving transformation of the
probability space $(\Omega, \Prob)$ and $(g_{n})$ be a sequence of
integrable functions satisfying 
\begin{equation*}
g_{n+m}(x)\leq g_{n}(x)+g_{m}(F^{n}x).
\end{equation*}
Then, almost surely, 
\begin{equation*}
\lim_{n\rightarrow \infty }\frac{g_{n}(x)}{n}=g(x)\geq -\infty ,
\end{equation*}
where $g(x)$ is an $F$-invariant function. Additionally, if $F$ is an
ergodic transformation, $g$ is constant for almost all $x\in \Omega $.
\end{proposition}

For a concise proof we refer the reader to Steele~\cite{Steele89}.

\begin{proposition}
\label{essdim} For a.a.\ $\omega $, the box-counting dimension and Hausdorff
dimension of $K_{\omega } $ is equal to 
\begin{equation*}
\lim_{n\rightarrow \infty }\frac{\log \left\Vert A_{\omega _{1}}\dots
A_{\omega _{n}}\right\Vert }{n\lvert\log r\rvert}.
\end{equation*}
\end{proposition}

\begin{proof}
For each $\gamma _{j}\in \mathcal{E}$ there is some finite word $\nu _{j}\in
\bigcup_{k}\mathcal{A}^{k}$ such that $\gamma _{j}$ is the characteristic
vector of the $\nu _{j}$-net interval $\Delta _{j}$ of level $N_{j}$. As in
the proof of the proposition above, $\left\Vert e_{j}A_{\omega _{1}}\dots
A_{\omega _{n}}\right\Vert $ is the number of $(\nu _{j}\omega )$ -net
subintervals that are descendants of $\Delta _{j}$ (under the word $(\nu
_{j}\omega )$) at level $N_{j}+n$.

Let $g_n(\omega)=\log\lVert e_jA_{\omega_1}A_{\omega_2}\dots
A_{\omega_n}\rVert$. We note that the usual matrix norm is
submultiplicative, thus $g_n$ is subadditive with respect to the ergodic
shift map $\pi(\omega_{1},\omega_2,\dots)=(\omega_2,\omega_3,\dots)$. An
application of Kingman's subadditive ergodic theorem shows that almost
surely 
\begin{equation*}
\lim_{n}\log \left\Vert e_{j}A_{\omega _{1}}\dots A_{\omega _{n}}\right\Vert
/n
\end{equation*}
exists.

As above, the dimensions of $\Delta _{j}\cap K_{\nu _{j}\omega }$ and $%
K_{\nu _{j}\omega }$ coincide, hence 
\begin{equation*}
\lim_{n\rightarrow \infty }\frac{\log\left\Vert e_{j}A_{\omega _{1}}\dots
A_{\omega _{n}}\right\Vert }{n\lvert\log r\rvert}=\dim _{B}\left( \Delta _{j}\cap
K_{\nu _{j}\omega }\right) =\dim _{B}K_{\nu _{j}\omega }.
\end{equation*}
There is an index $j$ such that for infinitely many $n$, 
\begin{equation*}
\left\Vert A_{\omega _{1}}\dots A_{\omega _{n}}\right\Vert _{\text{row},\max
}=\left\Vert e_{j}A_{\omega _{1}}\dots A_{\omega _{n}}\right\Vert .
\end{equation*}

Another application of Kingman's subadditive ergodic theorem shows that
almost surely $\lim_{n}\log \left\Vert A_{\omega _{1}}\dots A_{\omega
_{n}}\right\Vert /n$ exists. Since the maximum row sum norm is comparable to
the usual matrix norm, for this choice of $j$ we have, almost surely, 
\begin{equation*}
\lim_{n\rightarrow \infty }\frac{\log\left\Vert e_{j}A_{\omega _{1}}\dots
A_{\omega _{n}}\right\Vert }{n\lvert\log r\rvert}=\lim_{n\rightarrow \infty }\frac{%
\log\left\Vert A_{\omega _{1}}\dots A_{\omega _{n}}\right\Vert }{n\lvert\log r\rvert}.
\end{equation*}
Finally, we know that the upper box-counting dimension coincides with the
Hausdorff dimension of $K_{\omega }$ a.s.\ (see \cite{Troscheit17}), which
proves the claim.
\end{proof}

Given $\omega \in \Omega $, we call $x\in K_{\omega }$ an \textit{essential
point} if $x$ has $\omega $-symbolic representation $[x]_{\omega }=(\gamma
_{0},\gamma _{1},\dots)$ where there is some $J$ such that $\gamma _{j}\in 
\mathcal{E}$ for all $j\geq J$. As in the single IFS case, these points have
full $\mu _{\omega }$-measure for a.a.\ $\omega $.

\begin{proposition}
For a.a.\ $\omega $, the set $\mathcal{T}$ of non-essential points in $%
K_{\omega }$ has $\mu _{\omega }$-measure zero.
\end{proposition}

\begin{proof}
Here we use the fact that if $\Delta $ is an $\omega $-net interval and $%
\Delta ^{\prime }$ is a net subinterval $k$ levels lower, then $\mu _{\omega
}(\Delta ^{\prime })\geq p_{\min }^{k}\mu _{\omega }(\Delta )=\delta \mu
_{\omega }(\Delta )$ for a positive constant $\delta $ (depending on $k$).

By concatenating words, as needed, one can see that there is a finite word $%
\nu =(\nu _{1},\dots ,\nu _{k})$ such that given any characteristic vector $%
\gamma $, there is a $\nu $-admissible path linking $\gamma $ to an
essential characteristic vector. The set of full measure in $\Omega $ that
we take will be the $\omega ^{\prime }s$ for which $\nu $ appears as a
substring infinitely often.

Assume $\omega =(\omega _{1},\dots,\omega _{N},\nu _{1},\dots,\nu
_{k},\dots) $ and let $\Delta _{0}$ be any $\omega$-net interval of level $N$%
. At least one of the level $k$ descendants of $\Delta _{0}$ is an essential
net interval (i.e., its characteristic vector is essential), hence $\mu
_{\omega }(\Delta _{0}\cap \mathcal{T)}\leq (1-\delta )\mu _{\omega }(\Delta
_{0})$. Repeated application of this argument along the infinitely many
disjoint substrings of $\omega $ of the form $\nu $ shows that $\mu _{\omega
}(\Delta _{0}\cap \mathcal{T})\leq (1-\delta )^{n}\mu _{\omega }(\Delta
_{0}) $ for all $n$ and hence $\Delta _{0}\cap \mathcal{T}$ has measure
zero. As there are only finitely many $\omega$-net intervals at level $N$,
it follows that $\mu _{\omega }(\mathcal{T)}=0$.
\end{proof}

\begin{remark}
In the case of a deterministic IFS of finite type the essential class has
the property that the set of local dimensions at points in the essential
class is a closed interval. This was proven in \cite{HHM} and relied heavily
upon properties of `periodic points', which seem to have no analogue in the
random case. It would be interesting to know whether it was still true that
a similar conclusion holds for the local dimension theory in the random
case. In the next section (see Theorem \ref{CommutingThm}) we prove that the
(full) set of attainable local dimensions \textit{is} a closed interval for
a special family of examples of RIFS of finite type.
\end{remark}

\section{The Local Dimension Theory for a Commuting Case}

\label{sec:5}

\subsection{Local dimensions of commuting RIFS of finite type}

In this section we consider a special case when the RIFS is not only of
finite type, but also the transition matrices have a commuting-like
property, which we now explain. We continue to assume that $\mathrm{hull}%
(K_\omega) = [0,1]$ for all $\omega \in \Omega$.

\begin{definition}
Let $\fS$ be a RIFS of finite type. The finite word $\nu \in \mathcal{A}^{k}$
is called a \textbf{sink} if there is a reduced characteristic vector $\beta 
$ with the property that given any characteristic vector $\gamma $, there is
an admissible $\nu $-path linking $\gamma $ to $\beta $ and there is no
admissible $\nu $-path linking $\gamma $ to any other reduced characteristic
vector.

In this case, we also say that the RIFS has a sink and we call $\beta $ a 
\textbf{sink characteristic vector}.

We say the RIFS is \textbf{commuting} if the sink characteristic vector has
only one element in its neighbour set.
\end{definition}

We call this property `commuting' because it means that the transition
matrix of any admissible path linking a sink characteristic vector to itself
is a scalar. In Proposition \ref{ExCommuting} we exhibit a family of RIFS
that are commuting. See also Subsection \ref{CommExAnal} where we analyze a
specific example of a commuting RIFS in detail.

\begin{definition}
Suppose the RIFS has a sink $\nu $. Let $\omega \in \Omega $ be given and
denote by $N_{j}=N_{j}(\omega )$ the index of the last letter of the $j$'th
occurrence of $\nu $ as a substring in $\omega $. We call $N_{j}$ the $j$'th 
\textbf{neck level}.
\end{definition}

Observe that if $x\in K_{\omega }$ has symbolic representation $[x]_{\omega
}=(\gamma _{0},\dots,\gamma _{N_{1}},\dots,\gamma _{N_{2}},\dots)$, then $%
\gamma _{N_{j}}$ is the sink characteristic vector for each $j$.

Assume the sink $\nu $ is a word of length $k$. If we let $p_{0}=\mathbb{%
P\{\omega }:(\omega _{1},\dots,\omega _{k})=\nu \}$, then $p_{0}>0$. The
probability that none of the first $n$ blocks of $k$ letters is the word $%
\nu $ is $(1-p_{0})^{n}$. Thus, the first neck level has expectation 
\begin{equation*}
\E(N_{1}(\omega ))\leq \sum_{i=1}^{\infty }(1-p_{0})^{i-1}(k\,i)<\infty .
\end{equation*}
In particular, almost all $\omega $ have infinitely many neck levels, i.e.,
the sink recurs as a substring in $\omega $ infinitely often.

It was shown in Theorem \ref{regLocDim} that for regular measures, $\mu
_{\omega }$, associated with a RIFS of finite type, we have the local
dimension formula 
\begin{equation*}
\dimlc\mu _{\omega }(x)=\lim_{n\rightarrow \infty }\frac{\log \mu _{\omega
}(\Delta _{\omega ,n}(x))}{n\log r}=\lim_{n\rightarrow \infty }\frac{\log
\lVert T_\omega(\gamma _{0},\gamma _{1},\dots,\gamma _{n})\rVert }{n\log r}
\end{equation*}
when $x\in K_{\omega }$ has symbolic representation $[x]_{\omega }=(\gamma
_{0},\gamma _{1},\dots)$. Hence we are interested in studying the Lyapunov
exponents, \thinspace $\lim_{n}\log \lVert T_\omega(\gamma _{0},\gamma
_{1},\dots,\gamma _{n})\rVert /n$. We will be able to analyze these in the
commuting case. Thus for the remainder of this section we will assume that
the RIFS $\fS$ is commuting and that the probabilities $\{p_{j,k}\}$ are
chosen so that the random self-similar measures are regular.

Our first step will be to check that it is enough to study these limits
taken along the subsequence $(N_{j})$. This will be helpful since the
commuting property ensures that the `block' transition matrices along a path
from one sink characteristic vector to the next, 
\begin{equation*}
\mathbf{B}_{j}(\gamma )=\mathbf{B}_{\omega ,j}(\gamma ):=T_\omega(\gamma
_{N_{j-1}},\dots,\gamma _{N_{j}})\text{ for }\gamma =(\gamma _{j})\text{ }
\end{equation*}
(here $N_{0}=0$), are scalars and hence commute with all transition
matrices. We need a preliminary technical result.

\begin{lemma}
\label{lem:vanishingratio} Almost surely, the relative gaps between
successive necks become arbitrarily small, that is 
\begin{equation*}
\frac{N_{j+1}-N_{j}}{N_{j}}\rightarrow 0\quad \text{ a.s.}
\end{equation*}
\end{lemma}

\begin{proof}
Recall that $\pi $ is the shift map. Note that $N_{j}(\omega )=N_{1}(\pi
^{N_{j-1}(\omega )}(\omega ))$ and this is independent of $N_{1}(\omega )$
for every $j$. Thus, by Birkhoff's ergodic theorem, almost surely 
\begin{equation}
\lim_{j\rightarrow \infty }\frac{N_{j}(\omega )}{j}=\lim_{j\rightarrow
\infty }\frac{\sum_{i=1}^{j}(N_{i}(\omega )-N_{i-1}(\omega ))}{j}
=\lim_{j\rightarrow \infty }\frac{1}{j}\sum_{i=1}^{j}N_{1}(\pi
^{N_{i-1}}\omega )=\E(N_{1}),
\end{equation}
therefore 
\begin{equation}
\lim_{j\rightarrow \infty }\frac{N_{j+1}(\omega )-N_{j}(\omega )}{j}
=0=\lim_{j\rightarrow \infty }\frac{N_{j+1}(\omega )-N_{j}(\omega )}{%
N_{j}(\omega )}\quad \text{ a.s. }\qedhere
\end{equation}
\end{proof}

\begin{lemma}
Given $\omega \in \Omega $ and integer $n$, choose $j$ such that $%
N_{j}(\omega )<n\leq N_{j+1}(\omega )$. For almost all $\omega $, the
limiting behaviours of 
\begin{equation}
\frac{1}{n}\log \lVert T_\omega(\gamma _{0},\gamma _{1},\dots ,\gamma
_{n})\rVert \quad\text{and}\quad\frac{1}{N_{j}}\log \lVert \mathbf{B}%
_{1}(\gamma )\mathbf{B}_{2}(\gamma )\dots \mathbf{B}_{j}(\gamma )\rVert 
\text{ for }\gamma =(\gamma _{j})\text{ }  \notag
\end{equation}
coincide.
\end{lemma}

\begin{proof}
Submultiplicity of the matrix norm gives 
\begin{align}
\frac{1}{n}\log \lVert T_\omega(\gamma _{0},\gamma _{1},\dots \gamma
_{n})\rVert & \leq \frac{1}{N_{j}}\log \lVert \mathbf{B}_{1}(\gamma )\mathbf{%
B}_{2}(\gamma )\dots \mathbf{B}_{j}(\gamma )T_\omega({\gamma _{N_{j}}}\dots {%
\gamma _{n}})\rVert  \notag \\
& \leq \frac{1}{N_{j}}\log \left( \lVert \mathbf{B}_{1}(\gamma )\mathbf{B}%
_{2}(\gamma )\dots \mathbf{B}_{j}(\gamma )\rVert \lVert T_\omega({\gamma
_{N_{j}}}\dots {\gamma _{n}})\rVert \right)  \notag \\
& \leq \frac{1}{N_{j}}\log \lVert \mathbf{B}_{1}(\gamma )\mathbf{B}%
_{2}(\gamma )\dots \mathbf{B}_{j}(\gamma )\rVert +\frac{N_{j+1}-N_{j}}{N_{j}}%
\delta  \notag
\end{align}
for some $\delta >0$ that does not depend on the path. By Lemma~\ref%
{lem:vanishingratio} 
\begin{equation*}
\liminf \frac{1}{n}\log \lVert T_\omega(\gamma _{0},\gamma _{1},\dots
\gamma_{n})\rVert \leq \liminf \frac{1}{N_{j}}\log \lVert \mathbf{B}%
_{1}(\gamma ) \mathbf{B}_{2}(\gamma )\dots \mathbf{B}_{j}(\gamma )\rVert
\end{equation*}
(and similarly for the $\limsup $).

For the other direction, we note that as each $\mathbf{B}_{i}$ is a positive
scalar, 
\begin{equation*}
\frac{1}{N_{j+1}}\log \lVert \mathbf{B}_{1}(\gamma )\mathbf{B}_{2}(\gamma
)\dots \mathbf{B}_{j+1}(\gamma )\rVert =\frac{\sum_{i=1}^{j+1}\log \mathbf{B}%
_{i}(\gamma )}{N_{j+1}}
\end{equation*}
But $\mathbf{B}_{j+1}(\gamma )\leq \left( \max \left\Vert T_\omega(\xi ,\chi
)\right\Vert \right) ^{N_{j+1}-N_{j}}$ where the maximum is taken over all
parent/child characteristic vectors $(\zeta ,\chi )$. As there only finitely
many characteristic vectors this maximum is bounded, say by $C$. Hence 
\begin{equation*}
\left\vert \frac{\log \mathbf{B}_{j+1}}{N_{j+1}}\right\vert \leq C \frac{%
N_{j+1}-N_{j}}{N_{j}}\rightarrow 0.
\end{equation*}
As $N_{j+1}/N_{j}\rightarrow 1$, the limiting behaviours of 
\begin{equation*}
\frac{1}{N_{j}}\log \lVert \mathbf{B}_{1}(\gamma )\mathbf{B}_{2}(\gamma
)\dots \mathbf{B}_{j}(\gamma )\rVert \quad\text{and}\quad\frac{1}{N_{j+1}}%
\log \lVert \mathbf{B}_{1}(\gamma )\mathbf{B}_{2}(\gamma )\dots \mathbf{B}%
_{j+1}(\gamma )\rVert
\end{equation*}
coincide. By similar reasoning to the first part of the argument 
\begin{equation*}
\liminf \frac{1}{N_{j+1}}\log \lVert \mathbf{B}_{1}(\gamma ) \mathbf{B}%
_{2}(\gamma )\dots \mathbf{B}_{j+1}(\gamma )\rVert \leq \liminf \frac{1}{n}
\log \lVert T_\omega(\gamma _{0},\gamma _{1},\dots \gamma _{n})\rVert ,
\end{equation*}
(and similarly for the $\limsup $) and this completes the proof.
\end{proof}

For each $\omega \in \Omega $ and $i\in \mathbb{N}$ there are only finitely
many admissible paths $\gamma _{N_{i}},\dots,\gamma _{N_{i+1}}$ and hence
only finitely many choices for the matrices ${\mathbf{B}}_{i}(\gamma)$.
Consequently, for each $\omega $ and $i$ there is a choice that maximizes
(or minimizes) the logarithm of the norm (independent of the proceeding and
following block). We denote these maximal and minimal choices by $\overline{%
\mathbf{B}}_{i}(\omega)$ and $\underline{\mathbf{B}}_{i}(\omega)$,
respectively. Recalling that the block matrices (including $\underline{%
\mathbf{B}}_{i}$, $\overline{\mathbf{B}}_{i})$ are actually scalars, we see
that for almost all $\omega $ and all $\omega $-paths $\gamma $ we have 
\begin{eqnarray}
\lim_{j}\frac{1}{N_{j}}\log \lVert \underline{\mathbf{B}}_{1}\dots 
\underline{\mathbf{B}}_{j}\rVert &\leq &\underline{\lim}_{j}\frac{1}{N_{j}}%
\log \lVert {\mathbf{B}}_{1}(\gamma )\dots {\mathbf{B}}_{j}(\gamma )\rVert 
\text{ }  \label{minmax} \\
&\leq &\overline{\lim }_{j}\frac{1}{N_{j}}\log \lVert {\mathbf{B}}%
_{1}(\gamma )\dots {\mathbf{B}}_{j}(\gamma )\rVert  \notag \\
&\leq &\lim_{j}\frac{1}{N_{j}}\log \lVert \overline{\mathbf{B}}_{1}\dots 
\overline{\mathbf{B}}_{j}\rVert .  \notag
\end{eqnarray}
The `almost all' is required for the blocks to exist, but we could equally
have expressed this in terms of the original matrices. However, as we are
interested in generic behaviour, we will not make a distinction.

We can apply Kingman's subadditive ergodic theorem, (Proposition~\ref%
{prop:Kingman}), to our setting by considering the \emph{neck shift} $%
\pi^{N_{1}}$, that is 
\begin{equation*}
\pi ^{N_{1}}(\omega )=(\omega _{N_{1}(\omega )+1},\omega _{N_{1}+2},\dots ),
\end{equation*}
putting $\pi ^{N_{1}}(\omega )$ equal to the empty word for the measure zero
set of $\omega $'s where $N_{j}(\omega )=\infty $ for some $j$. The
independence of the necks guarantees the invariance of $\pi ^{N_{1}}$ with
respect to our Bernoulli probability measure $\Prob$. It is also easy to
check that $\pi ^{N_{1}}$ inherits ergodicity from the usual (one-letter)
shift on $\Omega $. Finally, letting 
\begin{equation*}
g_{n}(\omega )=\log \lVert \underline{\mathbf{B}}_{1}\dots \underline{ 
\mathbf{B}}_{n}\rVert =\log \underline{\mathbf{B}}_{1}\dots \underline{ 
\mathbf{B}}_{n}
\end{equation*}
we get the required subadditivity due to the submultiplicativity of the
matrix norm.

We can use Kingman's subadditive ergodic theorem to get explicit bounds on
the Lyapunov spectrum and hence the local dimensions.

\begin{proposition}
Almost surely, 
\begin{equation}
\lim_{j\rightarrow \infty }\frac{1}{N_{j}}\log (\underline{\mathbf{B}}%
_{1}\dots \underline{\mathbf{B}}_{j}) =\frac{\E\left( \log \underline{%
\mathbf{B}}_{1}\right) }{\E(N_{1})}  \notag
\end{equation}
and 
\begin{equation}
\lim_{j\rightarrow \infty }\frac{1}{N_{j}}\log (\overline{\mathbf{B}}
_{1}\dots \overline{\mathbf{B}}_{j})=\frac{\E\left( \log \overline{\mathbf{B}
}_{1}\right) }{\E(N_{1})}.  \notag
\end{equation}
\end{proposition}

\begin{proof}
Since $1\leq \E(N_{1})<\infty $ a.s., Birkhoff's ergodic theorem applied to
the neck shift gives 
\begin{equation}
\lim_{j}\frac{N_{j}}{j}=\lim_{j}\frac{1}{j}\sum_{i=1}^{j}(N_{i}-N_{i-1})=
\lim_{j}\frac{1}{j}\sum_{i=1}^{j}N_{1}(\pi ^{N_{i-1}}(\omega ))=\E(N_{1}). 
\notag
\end{equation}
Recall that if $B$ is a product of $N$ transition matrices, then $\left\Vert
B\right\Vert \geq \min p_{j,k}^{N}$. Thus Kingman's subadditive ergodic
theorem ensures there is some $\alpha \in \mathbb{R}$ such that 
\begin{equation*}
\lim_{j}\frac{1}{N_{j}}\log \underline{\mathbf{B}}_{1}\dots \underline{ 
\mathbf{B}}_{j}=\alpha\quad \text{ a.s.,}
\end{equation*}
so 
\begin{equation*}
\alpha =\lim_{j}\frac{\log \underline{\mathbf{B}}_{1}\dots \underline{ 
\mathbf{B}}_{j}}{N_{j}}=\left( \lim_{j}\frac{N_{j}}{j}\right) ^{-1}\lim_{j} 
\frac{\log \underline{\mathbf{B}}_{1}\dots \underline{\mathbf{B}}_{j}}{j}= 
\frac{\E\left( \log \underline{\mathbf{B}}_{1}\right) }{\E(N_{1})}\quad 
\text{ a.s. }
\end{equation*}
The almost sure finiteness of $\E(N_{1})$ also implies $\lvert \E(\log 
\underline{\mathbf{B}}_{1})\rvert <\infty $.

The second claim can be proved analogously.
\end{proof}

\begin{theorem}
\label{CommutingThm}Let $\fS$ be a RIFS of finite type and assume that the
associated random self-similar measures, $\mu _{\omega }$, are regular. If
the RIFS is commuting, then for almost all $\omega $, the set of attainable
local dimensions for the measure $\mu _{\omega }$ is the closed interval, 
\begin{equation*}
\left[ \frac{\E\left( \log \overline{\mathbf{B}}_1 \right)}{\E(N_1)\log r}%
,\, \frac{\E\left( \log \underline{\mathbf{B}}_1 \right)}{\E(N_1)\log r} %
\right].
\end{equation*}
\end{theorem}

\begin{proof}
To prove this, we construct blocks that interpolate between the maximal and
minimal achievable local dimensions. To be precise, given an $\omega $ we
chose the $j$'th block to be $\underline{\mathbf{B}}_{j}$ with probability $%
t $ and to be $\overline{\mathbf{B}}_{j}$ with probability $1-t$. We denote
this random variable by $X_{j}$ so that $\Prob_{1}\{X_{j}=\underline{\mathbf{%
\ B}}_{j}\}=t$ and $\Prob_{1}\{X_{j}=\overline{\mathbf{B}}_{j}\}=1-t$. We
consider the product measure of the original measure $\Prob$ on $\Omega $
with $\Prob_{1}$, the $(t,1-t)$-Bernoulli measure on $\{0,1\}^{\BbN}$. Since
the latter is strongly mixing, the product measure is ergodic with respect
to the shift map on $\Omega \times \{0,1\}^{\BbN}$ given by $F(\omega
,\lambda )=(\pi ^{N_{1}}(\omega ),\lambda _{2},\lambda _{3}\dots )$. Again
we can use Kingman's subadditive ergodic theorem to obtain 
\begin{equation}
\lim_{j}\frac{1}{N_{j}}\log X_{1}\dots X_{j}=\frac{\E\left( \log
X_{1}\right) }{ \E\left( N_{1}\right) }  \notag
\end{equation}
(where the expectation is with respect to the product measure). Further, 
\begin{equation}
\E\left( \log X_{1}\right) =t\E\left( \log \underline{\mathbf{B}}_{1}\right)
+(1-t)\E\left( \log \overline{\mathbf{B}}_{1}\right) .  \notag
\end{equation}
Varying $t\in \lbrack 0,1]$, it follows that for a.a.\ $\omega $ there
exists an admissible path $\gamma =(\gamma _{0},\gamma _{1},\dots )$ and $%
x\in K_{\omega }$ with $[x]_{\omega }=\gamma $, which has Lyapunov exponent
given by 
\begin{equation*}
\frac{t\E\left( \log \underline{\mathbf{B}}_{1}\right) +(1-t)\E\left( \log 
\overline{\mathbf{B}}_{1}\right) }{\E\left( N_{1}\right) }.
\end{equation*}
Hence the set of attainable local dimensions contains the interval 
\begin{equation*}
\left[ \frac{\E\left( \log \overline{\mathbf{B}}_{1}\right) }{\E(N_{1})\log
r }\;,\;\frac{\E\left( \log \underline{\mathbf{B}}_{1}\right) }{\E%
(N_{1})\log r }\right] .
\end{equation*}

The fact that the set of local dimensions is contained within this interval
for a.a.\ $\omega $ is clear from the minimality/maximality choices of $%
\underline{\mathbf{B}}_{1},\overline{\mathbf{B}}_{1}$ (see equation %
\eqref{minmax}) and Kingman's subadditive ergodic theorem.
\end{proof}

\subsection{A family of commuting RIFS}

In this subsection we will construct a family of examples of commuting RIFS.
We speculate that these are the only examples under the additional
assumption that the self-similar sets are always $[0,1]$ and provide some
evidence to support this speculation.

\begin{proposition}
\label{ExCommuting}Let $d\geq 2$ be an integer. Assume that $\fS_{1}$
consists of the $d$ contractions $S_{1,i}(x)=x/d+i/d$ for $i=0,\dots ,d-1$
and that the other IFS, $\fS_{j}$, $j=2,\dots,m$, consist of contractions of
the form $S(x)=x/d+a/d^{k}$ for choices of $a,k\in \mathbb{N}$. Finally,
assume each $\fS_{j}$ has self-similar set equal to $[0,1]$ (hence all $%
K_{\omega }=[0,1]$). Then the RIFS $\fS=\{\fS_1, \fS_2, \dots, \fS_m\}$ is
commuting.
\end{proposition}

\begin{proof}
Let $N$ be the maximum exponent $k$ taken over all the contractions in all
the $\fS_{j}$. Note that all images of $S_{\omega ,\sigma }([0,1])$, where $%
\sigma \in \Lambda _{\omega ,n}$, will be of the form $%
[c/d^{n+N},(c+d^{N})/d^{n+N}]$ for suitable positive integers $c,N$. Hence
all level $n$ net intervals will be of the form $[a/d^{n+N},b/d^{n+N}]$ for $%
a,b\in \mathbb{N}$.

Let $\nu =(1,1,\dots,1)\in \mathcal{A}^{N}$. Acting on these intervals by
any contraction $S_{\nu ,\sigma }$ (a composition of $N$ contractions from
the first IFS) gives intervals again of the form $[a^{\prime
}/d^{n+N},b^{\prime }/d^{n+N}]$ for suitable $a^{\prime }$, $b^{\prime }\in 
\mathbb{N}$. As the level $n$ net intervals have length at most $1/d^{n}$,
these have length at most $1/d^{N+n}$ and hence we must have $b^{\prime
}=a^{\prime }+1$.

Thus these are all net intervals and they all have the same reduced
characteristic vector as $[0,1]$. This proves $\nu $ is a sink and the RIFS
is commuting.
\end{proof}

\begin{proposition}
\label{commchar}Assume the RIFS is commuting with sink $\nu
=(1,1,\dots,1)\in \mathcal{A}^{k}$ and that the self-similar set associated
with $\fS_{1}$ is $[0,1]$. Then $\fS_{1}$ consists of the $d$ contractions $%
S_{1,j}(x)=x/d+j/d$ for $j=0,\dots ,d-1$, where $d$ is an integer, $d\geq 2$.
\end{proposition}

\begin{proof}
Suppose the sink characteristic vector is $\beta $. Observe that the action
of $\fS_{1}$ must send $\beta $ to itself, for if it mapped $\beta $ to a
different reduced characteristic vector, say $\gamma _{1}$, then the
admissible path that links $\beta $ to itself in $k$ iterations of $\fS_{1}$
, say $\beta ,\gamma _{1},\dots,\gamma _{k-2},\beta $, would link $\gamma
_{2}$ to $\gamma _{1}$ in $k$ steps and that contradicts the definition of a
sink. So we can assume $\beta $ is mapped to $d$ (reduced) characteristic
vectors $\beta $ under any (one) contraction in $\fS_{1}$ and to no other
characteristic vector. We must have $d\geq 2$ since net interval lengths
eventually decrease.

Let $\Delta $ be a level $n$ $\omega $-net interval with characteristic
vector $\beta $. The length of $\Delta $ is $cr^{n}$ for some constant $c$,
where $r$ is the common contraction factor of the RIFS. Each child of $%
\Delta $ also has characteristic vector $\beta $ and hence has length $%
cr^{n+1}$. As there are $d$ children we must have $r=1/d$.

As the RIFS is commuting, $\beta $ can have only one neighbour. The same is
true for all its children, with even the same (normalized) value. Suppose $%
S_{\omega ,\sigma }([0,1])$ is the only set covering $\Delta $ with $\sigma $
of length $n$. The children of $\Delta $ are generated by the sets $%
S_{\omega 1,\sigma j}([0,1])$ where $j\in \{0,\dots,\mathcal{N}(1)\}$.
Consideration of the geometry of neighbours shows that these sets must not
overlap, but rather, be precisely adjacent. That proves the maps $S_{1,j}$
are as claimed.
\end{proof}

\begin{remark}
When the sink has the form $(j,j,\dots,j)$, then, as observed in the proof
above, the sink characteristic vector can only map to itself under $\fS_{j}$%
. Consideration of the geometry shows that it must be the characteristic
vector of the initial net interval. Since the initial characteristic vector
has every characteristic vector as a descendant, the essential class must be
everything.
\end{remark}

\section{Examples}

\label{sec:6}

\subsection{Detailed analysis of a commuting RIFS\label{CommExAnal}}

In this subsection we will consider a specific example of a commuting RIFS
and find the minimum and maximum values for the local dimensions of the
associated self-similar measures.

\subsubsection{Set up of the example}

Consider the two iterated function systems:

\begin{enumerate}
\item $\fS_{1}$ with $S_{1,1}(x)=x/3$, $S_{1,2}(x)=x/3+1/6$, $%
S_{1,3}(x)=x/3+2/3$ and $p_{1,1}=p_{1,3}=1/6$, $p_{1,2}=2/3;$

\item $\fS_{2}$ with $S_{2,1}(x)=x/3$, $S_{2,2}(x)=x/3+1/9$, $%
S_{2,3}(x)=x/3+1/3$, $S_{2,4}(x)=x/3+2/3$ and $p_{2,1}=p_{2,2}=p_{2,4}=1/6$, 
$p_{2,3}=1/2$.
\end{enumerate}

This RIFS $\fS=(\{\fS_{1},\fS_{2}\},\mathbb{P})$, for any choice of
probability measure $\mathbb{P=(}\theta ,1-\theta )$, is commuting being a
special case of Proposition \ref{ExCommuting}. There are 5 reduced
characteristic vectors: 
\begin{align*}
CV_{1} &=(1,(0))\text{; } & CV_{2}&=(2/3,(1/3))\text{; } & 
CV_{3}=(2/3,(0,1/3)) \text{; } \\
CV_{4} &=(1/3,(0))\text{; } & CV_{5}&=(1/3,(0,2/3)). & 
\end{align*}

Figure \ref{fig:Pic0} shows the parent/child relationships under $\fS_{1}$
and Figure \ref{fig:Pic1} shows this under $\fS_{2}$. It can be seen from
these graphs that the sink is the word $(1)$ and the sink characteristic
vector is $CV_{1}$. The transition matrices under the action of $\fS_{1}$
are given by:

\begin{itemize}
\item $CV_{1}\rightarrow CV_{1},CV_{1},CV_{1}$, 
\begin{equation*}
\begin{bmatrix}
1/6%
\end{bmatrix}%
, \hspace{0.2 in} 
\begin{bmatrix}
2/3%
\end{bmatrix}%
, \hspace{0.2 in} 
\begin{bmatrix}
1/6%
\end{bmatrix}%
,
\end{equation*}

\item $CV_{2}\rightarrow CV_{1},CV_{1},$ 
\begin{equation*}
\begin{bmatrix}
2/3%
\end{bmatrix}%
, \hspace{0.2 in} 
\begin{bmatrix}
1/6%
\end{bmatrix}%
,
\end{equation*}

\item $CV_{3}\rightarrow CV_{1},CV_{1},$ 
\begin{equation*}
\begin{bmatrix}
2/3 \\ 
1/6%
\end{bmatrix}%
, \hspace{0.2 in} 
\begin{bmatrix}
1/6 \\ 
2/3%
\end{bmatrix}%
,
\end{equation*}

\item $CV_{4}\rightarrow CV_{1}$, 
\begin{equation*}
\begin{bmatrix}
1/6%
\end{bmatrix}%
,
\end{equation*}

\item $CV_{5}\rightarrow CV_{1}$, 
\begin{equation*}
\begin{bmatrix}
1/6 \\ 
1/6%
\end{bmatrix}%
.
\end{equation*}
\end{itemize}

\begin{figure}[tbp]
\includegraphics[scale=0.5]{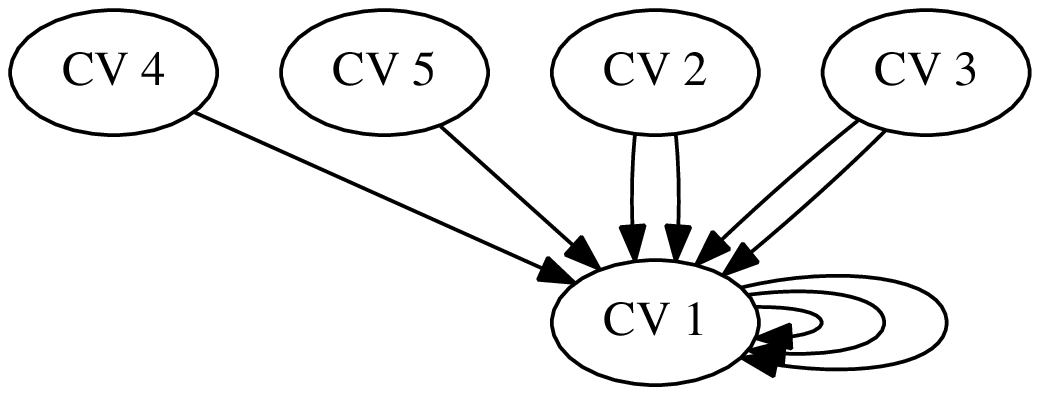}
\caption{$\fS_1$}
\label{fig:Pic0}
\end{figure}
Under $\fS_{2}$ the transition matrices are:

\begin{itemize}
\item $CV_{1}\rightarrow CV_{4},CV_{3},CV_{5},CV_{2},CV_{1},$ 
\begin{equation*}
\begin{bmatrix}
1/6%
\end{bmatrix}%
, \hspace{0.2 in} 
\begin{bmatrix}
1/6 & 1/6%
\end{bmatrix}%
, \hspace{0.2 in} 
\begin{bmatrix}
1/6 & 1/2%
\end{bmatrix}%
, \hspace{0.2 in} 
\begin{bmatrix}
1/2%
\end{bmatrix}%
, \hspace{0.2 in} 
\begin{bmatrix}
1/6%
\end{bmatrix}%
,
\end{equation*}

\item $CV_{2}\rightarrow CV_{5},CV_{2},CV_{1},$ 
\begin{equation*}
\begin{bmatrix}
1/6 & 1/2%
\end{bmatrix}
\hspace{0.2 in} 
\begin{bmatrix}
1/2%
\end{bmatrix}%
, \hspace{0.2 in} 
\begin{bmatrix}
1/6%
\end{bmatrix}%
,
\end{equation*}

\item $CV_{3}\rightarrow CV_{5},CV_{3},CV_{5},CV_{2},$ 
\begin{equation*}
\begin{bmatrix}
1/6 & 1/2 \\ 
0 & 1/6%
\end{bmatrix}%
, \hspace{0.2 in} 
\begin{bmatrix}
1/2 & 0 \\ 
1/6 & 1/6%
\end{bmatrix}%
, \hspace{0.2 in} 
\begin{bmatrix}
0 & 1/6 \\ 
1/6 & 1/2%
\end{bmatrix}%
, \hspace{0.2 in} 
\begin{bmatrix}
1/6 \\ 
1/2%
\end{bmatrix}%
,
\end{equation*}

\item $CV_{4}\rightarrow CV_{4},CV_{3},$ 
\begin{equation*}
\begin{bmatrix}
1/6%
\end{bmatrix}%
, \hspace{0.2 in} 
\begin{bmatrix}
1/6 & 1/6%
\end{bmatrix}%
,
\end{equation*}

\item $CV_{5}\rightarrow CV_{4},CV_{3},$ 
\begin{equation*}
\begin{bmatrix}
1/6 \\ 
1/6%
\end{bmatrix}%
, \hspace{0.2 in} 
\begin{bmatrix}
1/6 & 0 \\ 
1/6 & 1/6%
\end{bmatrix}%
.
\end{equation*}
\end{itemize}

\begin{figure}[tbp]
\includegraphics[scale=0.5]{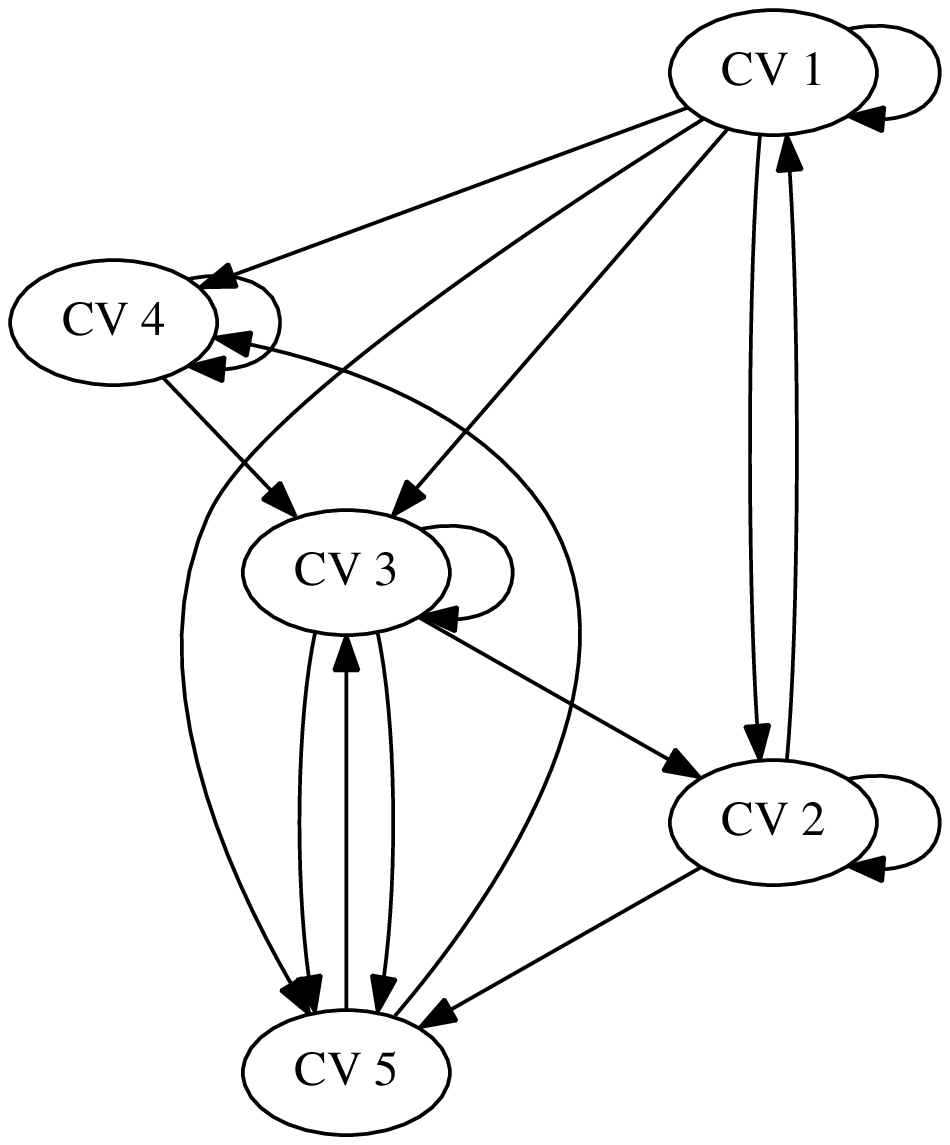}
\caption{$\fS_2$}
\label{fig:Pic1}
\end{figure}

\subsubsection{Dimensional analysis}

In order to find bounds on the upper and lower local dimensions for the
associated self-similar measures, we will need to study the transition
matrices from one neck level to the next. (Of course, these matrix products
are positive scalars and hence coincide with their spectral radius.) These
arise from paths of varying lengths, which we refer to as a neck length. If
the neck length is $1$, the corresponding transition matrix is either the $%
1\times 1$ matrix $[1/6]$ or $[2/3]$; these come from the maps $%
CV_{1}\rightarrow CV_{1}$ under $\fS_{1}$, taking either the first or third
(non-reduced) characteristic vector, or the second, respectively.

More generally, we claim that the minimal spectral radius of a block
transition matrix corresponding to a neck of length $n$ is $6^{-n}$. We can
achieve this spectral radius by taking the path $CV_{1}\rightarrow
_{2}CV_{1}\rightarrow _{2}CV_{1}\dots CV_{1}\rightarrow _{1}CV_{1}$. (Here
the notation $CV_{1}\rightarrow _{2}CV_{1}$, for example, means that we
transition from the parent $CV_{1}$ to the child $CV_{1}$ using $\fS_{2}$.)
To see that this spectral radius is in fact minimal, simply observe that the
minimal row sum of all matrices in both $\fS_{1}$ and $\fS_{2}$ is bounded
below by $1/6$.

Next, we claim that the maximal spectral radius of a block transition matrix
corresponding to a neck of length $n$ is equal to $2^{2-n}/3$. We can
achieve this spectral radius by the path $CV_{1}\rightarrow
_{2}CV_{2}\rightarrow _{2}CV_{2}\rightarrow _{2}CV_{2}\dots
CV_{2}\rightarrow _{1}CV_{1}$.

To see that this spectral radius is maximal is much more involved than to
derive the minimum, and requires us to adapt techniques used to compute the
joint spectral radius of a set of matrices. See \cite{GuglielmiProtasov13}
for example.

We first claim that $2^{1-n}$ is the maximal spectral radius for any
(admissible) product of $n-1$ transition matrices using $\fS_{2}$. We will
write $T_{i,j}^{(2)}$ for the primitive transition matrix $%
T_{2}(CV_{i},CV_{j})$ and let $\widetilde{T}_{i,j}^{(2)}=2T_{i,j}^{(2)}$,
the `normalized' transition matrix. The claim is equivalent to the statement
that the spectral radius of a product of $n-1$ normalized transition
matrices, $\widetilde{T}_{i,j}^{(2)}$, is bounded. We can prove this by
establishing that there exist convex compact sets $K_{i}\subseteq \mathbb{R}$
for $i=1,2,4$ and $K_{3},K_{5}\subseteq \mathbb{R}^{2}$ such that $v_{i}%
\widetilde{T}_{i,j}^{(2)}\in K_{j}$ for all $v_{i}\in K_{i}$ and all
admissible pairs $(i,j)$. For this, we take%
\begin{equation*}
K_{1}=K_{2}=K_{4}=[0,1],
\end{equation*}
\begin{equation*}
K_{3}=\mathrm{hull}\left( [0,0],\left[0,\frac{1}{3}\right],[1,0],\left[\frac{%
1}{2},\frac{1}{3}\right]\right)
\end{equation*}
and 
\begin{equation*}
K_{5}=\mathrm{hull}\left( [0,0],\left[0,\frac{1}{3}\right],\left[\frac{1}{3}%
,0\right],\left[\frac{1}{3} ,1\right]\right) .
\end{equation*}
Note, for example, that 
\begin{equation*}
v_{3}\widetilde{T}_{3,5}^{(2)}=v_{3}%
\begin{bmatrix}
1/3 & 1 \\ 
0 & 1/3%
\end{bmatrix}%
\in \left\{ [0,0],\left[0,\frac{1}{9}\right],\left[\frac{1}{3},1\right],%
\left[\frac{1}{6},\frac{ 11}{18}\right]\right\} \subseteq K_{5}
\end{equation*}%
whenever $v_{3}$ is an extreme point of $K_{3}$ and by linearity this is
enough to show $v_{3}\widetilde{T}_{3,5}^{(2)}\in K_{5}$ whenever $v_{3}\in
K_{3}$. We leave the verification of the other cases to the interested
reader.

Any block transition matrix corresponding to a neck of length $n$, $B_{n}$,
will be of the form $T_{1,i_{2}}^{(2)}T_{i_{2},i_{3}}^{(2)}\dots
T_{i_{n-2},i_{n-1}}^{(2)}T_{i_{n-1},1}^{(1)}$ for suitable indices $%
i_{2},\dots,i_{n-1}$. Our previous remarks imply that $%
T_{1,i_{2}}^{(2)}T_{i_{2},i_{3}}^{(2)}\dots T_{i_{n-2},i_{n-1}}^{(2)}\in
2^{1-n}K_{i_{n-1}}$. Moreover, it can be checked that $K_{i}T_{i,1}^{(1)}=%
\begin{bmatrix}
c_{i}%
\end{bmatrix}%
$ where $0\leq c_{i}\leq 2/3$. Thus $B_{n}$ has spectral radius bounded by $%
2^{2-n}/3$.

With these facts and Theorem \ref{CommutingThm} we can determine the almost
sure interval of local dimensions. The maximal local dimension is
independent of the choice of $\mathbb{P}$. It comes from paths giving rise
to commuting blocks of length $n$ of spectral radius $\left( 1/6\right) ^{n}$
and therefore has value 
\begin{equation*}
\frac{\log 1/6}{\log 1/3}=1+\frac{\log 2}{\log 3}.
\end{equation*}

The minimal local dimension of the self-similar measure $\mu _{\omega }$
depends on $\mathbb{P}$. In the notation of Theorem \ref{CommutingThm}, the
minimum value is almost surely%
\begin{equation*}
\frac{\E\left( \log \overline{\mathbf{B}}_{1}\right) }{\E(N_{1})\log r}.
\end{equation*}%
Given $\theta =$ the probability of choosing $\fS_{1}$, the probability of a
neck of length $n$ is $\theta (1-\theta )^{n-1}.$ Any corresponding block
transition matrix $B_{n}$ has maximal spectral radius $2^{2-n}/3$. Thus 
\begin{equation*}
\E(N_{1})=\sum_{n=1}^{\infty }n\theta (1-\theta )^{n}
\end{equation*}%
and 
\begin{equation*}
\E\left( \log \overline{\mathbf{B}}_{1}\right) =\sum_{n=1}^{\infty }\theta
(1-\theta )^{n}\log \left( \frac{2^{2-n}}{3}\right) .
\end{equation*}%
Hence the minimal local dimension is equal to%
\begin{equation*}
\frac{\sum_{n=1}^{\infty }\theta (1-\theta )^{n}\log \left( \frac{2^{2-n}}{3}%
\right) }{\sum_{n=1}^{\infty }n\theta (1-\theta )^{n}\log 1/3}=\frac{\theta
\left( \log 3/4\right) +\log 2}{\log 3}\quad \text{ a.s.}
\end{equation*}%
For instance, if $\theta =1/2$, the minimum local dimension is almost surely 
$1/2$. As $\theta \rightarrow 0$, the minimum local dimension tends to $\log
2/\log 3$ and as $\theta \rightarrow 1$, it tends to $1-\log 2/\log 3$.

\subsection{Biased random Bernoulli convolution example}

In this subsection we will study the RIFS $\fS=\{\fS_{1},\dots,\fS_{m}\}$,
chosen with probabilities $\mathcal{P}=\{\theta _{j}\}$, where each $\theta
_{j}>0$ and each $\fS_{j}$ consists of the two contractions $S_{j,0}(x)=rx$, 
$S_{j,1}(x)=rx+1-r$ and probabilities $p_{j,0}=p_{j}$, $p_{j,1}=1-p_{j}$. We
will assume $r$ is the inverse of a simple Pisot number. Recall $r$ is a
simple Pisot number if it is the positive real root greater than $1$ of $%
r^{\ell} - r^{\ell -1} - \dots - r - 1$ for some $\ell \geq 2$. An example
is the golden mean.

In the case of a single IFS, $\fS_{1}$, it was shown in \cite{HHN} that if $%
p_{1}<1-p_{1}$, then the set of local dimensions of the associated
self-similar measure had an isolated point at $x=0$. Here we show a similar
result: If $p_i \leq 1-p_i$ for all $i$ and $p_{j}<1-p_{j}$ for some $j$,
then $\{\dimlc\mu _{\omega}(x):x\in \lbrack 0,1]\}$ admits an isolated point
at $x=0$ for a.a.\ $\omega $.

To see this, we argue as follows. First, note it is easy to see that for
a.a. $\omega $, 
\begin{equation*}
\dimlc\mu _{\omega }(0)=\lim_{n\rightarrow \infty }\frac{\log p_{\omega
_{1}}\cdot \cdot \cdot p_{\omega _{n}}}{n\log r}=\frac{\sum_{j=1}^{m}\theta
_{j}\log p_{j}}{\log r}.
\end{equation*}
The same proof as given in \cite[Lemma 4.2]{HHN} shows that there is an
integer $N$ such that each $x\in (0,1)$ has $\omega $-symbolic
representation $[x]=(\gamma ,\eta _{1},\eta _{2},\dots)$, where $\gamma $ is
an initial path, $\eta _{j}$ are essential $\omega $-paths of length at most 
$N$ and 
\begin{eqnarray*}
\left\Vert T_{\omega }(\eta _{2j-1},\eta _{2j},\eta _{2j+1,1})\right\Vert _{ 
\text{row,}\min } &\geq &p_{\omega _{K_{j-1}+1}}\cdot \cdot \cdot p_{\omega
_{K_{j}-1}}(1-p_{\omega _{K_{j}}}) \\
&=&\prod_{i=1}^{m}p_{i}^{s(i,j)}\prod_{i=1}^{m}(1-p_{i})^{t(i,j)}.
\end{eqnarray*}
Here $K_{j}=\sum_{i=1}^{j}L_{i}$ with $L_{i}=$ length $\eta _{2i-1}+$length $%
\eta _{2i}$, 
\begin{equation*}
s(i,j)=\card\{\ell :K_{j-1}<\ell <K_{j}\text{ and }\omega _{\ell }=i\}
\end{equation*}
and $t(i,j)=1$ if $\omega _{K_{j}}=i$ and $0$ else. Put $S_{J}(i)=
\sum_{j=1}^{J}s(i,j)$ and $T_{J}(i)=\sum_{j=1}^{J}t(i,j)$, so 
\begin{equation*}
S_{J}(i)=\card\{\ell :1\leq \ell \leq K_{J}\text{, }\ell \neq K_{j}\text{
for }j=1,\dots,J\text{ and }\omega _{\ell }=i\}
\end{equation*}
and $T_{J}(i)=\card\{K_{j}:j=1,\dots,J$ and $\omega _{K_{j}}=i\}$. Thus 
\begin{eqnarray*}
\frac{\log \left\Vert T_{\omega }(\eta _{1},\dots,\eta _{2J+1,1})\right\Vert
_{ \text{row,}\min }}{\sum_{i=1}^{J}L_{i}} &\geq &\frac{\log
\prod_{i=1}^{m}p_{i}^{\sum_{j=1}^{J}s(i,j)}(1-p_{i})^{\sum_{j=1}^{J}t(i,j)}}{
\sum_{i=1}^{J}L_{i}} \\
&\geq &\frac{\sum_{i=1}^{m}S_{J}(i)\log p_{i}+T_{J}(i)\log (1-p_{i})}{
\sum_{i=1}^{J}L_{i}}.
\end{eqnarray*}
For a.a.\ $\omega $, $S_{J}(i)/(K_{J}-J)\rightarrow \theta _{i}$ and $%
T_{J}(i)/J\rightarrow \theta _{i}$ as $J\rightarrow \infty $, thus 
\begin{eqnarray*}
&&\underline{\lim }_{J}\frac{\sum_{i=1}^{m}S_{J}(i)\log p_{i}+T_{J}(i)\log
(1-p_{i})}{\sum_{i=1}^{J}L_{i}} \\
&=&\sum_{i=1}^{m}\theta _{i}\left( \log p_{i}\underline{\lim }_{J}\frac{
K_{J}-J}{K_{J}}+\log (1-p_{i})\underline{\lim }_{J}\frac{J}{K_{J}}\right) \\
&=&\sum_{i=1}^{m}\theta _{i}\log p_{i}+\sum_{i=1}^{m}\theta _{i}\left( \log
(1-p_{i})+\log p_{i}\right) \underline{\lim }_{J}\frac{J}{K_{J}}.
\end{eqnarray*}
As $L_{i}\leq 2N$, it follows that $\underline{\lim }_{J}J/K_{J}\geq 1/2N$
and therefore for any $x\in (0,1)$ and a.a.\ $\omega $, 
\begin{eqnarray*}
\dimulc\mu _{\omega }(x) &=&\overline{\lim }_{J}\frac{\log \left\Vert
T_{\omega }(\eta _{1},\dots,\eta _{2J+1,1})\right\Vert _{\text{row,}\min }}{
K_{J}\log r} \\
&\leq &\frac{\sum_{i=1}^{m}\theta _{i}\log p_{i}}{\log r}+\frac{
\sum_{i=1}^{m}\theta _{i}(\log (1-p_{i})+\log p_{i})}{2N\log r}.
\end{eqnarray*}
If some $p_{i}\neq 1-p_{i}$, this is bounded (below) away from $\dimlc\mu
_{\omega }(0)$.

\subsection{Dimension of the essential class example}

Consider the two iterated function systems

\begin{enumerate}
\item $\fS_1$ with $S_{1,1}(x) = x/4$, $S_{1,2}(x) = x/4+1/6$, $S_{1,3}(x) =
x/4+7/12$ and $S_{1,4}(x) = x/4+3/4$.

\item $\fS_{2}$ with $S_{2,1}(x)=x/4$, $S_{2,2}(x)=x/4+1/18$, $%
S_{2,3}(x)=x/4+25/36$, $S_{2,4}(x)=x/4+3/4$,
\end{enumerate}

One can see from \cite{NW} and Proposition \ref{prop:4.6} that the RIFS $\{%
\fS_{1},\fS_{1}\}$ is of finite type. Using the computer, we have determined
that there are 492 reduced characteristic vectors, but only one reduced
characteristic vector in the essential class. This is the vector 
\begin{equation*}
(1/9,(0,1/9,2/9,3/9,4/9,5/9,6/9,7/9,8/9)).
\end{equation*}%
This vector maps to four copies of itself and hence Proposition \ref{essdim}
implies that the dimension of the RIFS is equal to $1$ a.s. In fact in this
case, the result is stronger than this. The dimension of the RIFS is equal
to $1$ for all choices of $\omega \in \Omega$. It is worth observing that $%
K_\omega$ does not have full support for any $\omega \in \Omega$.

It is interesting to note that if we consider the IFS $\fS_{1}$ alone, there
are 11 reduced characteristic vectors, and one reduced characteristic vector
in the essential class, namely $(1/3,(0,1/3,2/3))$. Similarly, $\fS_{2}$
(alone), has 117 reduced characteristic vectors and one reduced essential
characteristic vector, the same vector as for the RIFS. It can be shown that
the dimension of both $K_{1}$ and $K_{2}$ are also one.

\end{document}